\documentclass[12pt]{article}
\usepackage[top=1in,bottom=1in,left=1in,right=1in]{geometry}
\usepackage{indentfirst}
\usepackage{amsmath,amssymb,amsthm,mathtools,booktabs,array,longtable,microtype}
\usepackage{enumitem}
\usepackage{needspace,etoolbox}
\BeforeBeginEnvironment{theorem}{\Needspace{6\baselineskip}}
\BeforeBeginEnvironment{lemma}{\Needspace{6\baselineskip}}
\BeforeBeginEnvironment{proposition}{\Needspace{6\baselineskip}}
\BeforeBeginEnvironment{corollary}{\Needspace{6\baselineskip}}
\pretocmd{\subsubsection}{\Needspace{9\baselineskip}}{}{}
\usepackage{cite}
\usepackage[hidelinks]{hyperref}
\hypersetup{
  pdftitle={Critical Pairs for Mixed Restricted Sumsets in Prime Fields},
  pdfauthor={Weilin Zhang and Hongjian Li}
}

\newtheorem{theorem}{Theorem}[section]
\newtheorem{lemma}{Lemma}[section]
\newtheorem{proposition}{Proposition}[section]
\newtheorem{corollary}[lemma]{Corollary}

\theoremstyle{remark}

\newcommand{\F}{\mathbb{F}_p}
\newcommand{\rs}{\mathbin{\widehat+}}
\newcommand{\fall}[2]{(#1)_{\underline{#2}}}
\newcommand{\rise}[2]{(#1)_{\overline{#2}}}
\newcommand{\ind}{\mathbf 1}

\allowdisplaybreaks[2]
\usepackage{caption}

\begin{document}
\title{Critical Pairs for Mixed Restricted Sumsets\\ in Prime Fields}
\author{
Weilin Zhang$^{1}$\footnote{E-mail: weilin@gzhu.edu.cn.}\quad
Hongjian Li$^{2}$\footnote{Corresponding author. E-mail: lhj@gdufs.edu.cn. Supported by the Project of Guangdong University of Foreign Studies (Grant No. 2024RC063).}\\
{\small\itshape $^{1}$School of Mathematics and Information Science, Guangzhou University,}\\
{\small\itshape Guangzhou 510006, Guangdong, P. R. China}\\
{\small\itshape $^{2}$School of Mathematics and Statistics, Guangdong University of Foreign Studies,}\\
{\small\itshape Guangzhou 510006, Guangdong, P. R. China}
}

\date{}
\maketitle

\begin{abstract}
Let $p$ be an odd prime and let $A,B\subseteq\F$ be nonempty and distinct,
with $|A|\ge |B|$.  We give a complete classification of the pairs for
which the mixed restricted sumset
$A\rs B=\{x+y:x\in A,\ y\in B,\ x\ne y\}$ has size
$\min\{p,|A|+|B|-2\}$.  The main new contribution concerns the interior range $|A|+|B|\le p-1$.
For $|B|\ge3$ and $|A|-|B|\ge3$, criticality forces $A$ and $B$ to be common-endpoint arithmetic progressions, apart from a single common affine orbit with $(p,|A|,|B|)=(13,7,4)$.
For $|A|-|B|\in\{1,2\}$, criticality itself forces $B\subsetneq A$, after which
the restricted self-sum determines the deletion patterns. Together with the remaining cases, this yields the complete classification.
\end{abstract}

\medskip


\section{Introduction}\label{intro}

Let $p$ be an odd prime, and write $\F$ for the field with $p$ elements.
We write $\F^*=\F\setminus\{0\}$ and
$S^c=\F\setminus S$ for $S\subseteq\F$.
A basic inverse problem in additive combinatorics is to determine the
structure of sets whose sumset attains a sharp lower bound.  For ordinary
addition, the Cauchy--Davenport theorem \cite{Cauchy,Davenport} gives
\[
 |A+B|\ge\min\{p,|A|+|B|-1\}
\]
for nonempty $A,B\subseteq\F$.  Vosper's theorem
\cite{Vosper1,Vosper2} describes equality away from the boundary: if
$|A|,|B|\ge2$ and $|A+B|=|A|+|B|-1\le p-2$, then $A$ and $B$ are
arithmetic progressions with a common difference.

Restricted addition excludes representations with equal summands.  For
nonempty $A,B\subseteq\F$, define
\[
 A\rs B=\{x+y:x\in A,\ y\in B,\ x\ne y\}.
\]
The condition $x\ne y$ concerns the field elements themselves; excluding
$x+x$ need not remove the sum $2x$ if it has another representation.
For $A=B$, the Erd\H{o}s--Heilbronn bound, proved by Dias da Silva and
Hamidoune \cite{DH}, is
\[
 |A\rs A|\ge\min\{p,2|A|-3\}.
\]
K\'arolyi \cite[Theorem 7]{Karolyi} proved the corresponding inverse theorem: when
$|A|\ge5$ and $2|A|-3<p$, equality at $2|A|-3$ holds precisely for
arithmetic progressions.

For mixed addition, Alon, Nathanson, and Ruzsa \cite{ANR95,ANR96}
proved the sharp bound for unequal cardinalities, and K\'arolyi
\cite{KarolyiExceptional} treated distinct sets of equal cardinality
in the nonsaturated range.  Together with the
elementary saturated case, these results give
\begin{equation}\label{anr}
 |A\rs B|\ge\min\{p,|A|+|B|-2\}\qquad(A\ne B).
\end{equation}
Alon, Nathanson, and Ruzsa also posed the inverse problem of describing
the pairs attaining this bound \cite[Section 5, item 4]{ANR96}.  The self-sum inverse theorem
does not answer this question, since the sets need not coincide and the
extremal cardinality is different.

Throughout this paper, \emph{mixed} means $A\ne B$, not necessarily
$|A|\ne|B|$.  Interchanging the sets if necessary, put
$a=|A|\ge b=|B|\ge1$.  We call $(A,B)$ \emph{critical} if
\begin{equation}\label{critical}
 |A\rs B|=\min\{p,a+b-2\}.
\end{equation}
Henceforth, a \emph{critical pair} always means a mixed pair satisfying \eqref{critical}.
Our aim is a complete classification of these pairs.  The natural
equivalence is the common affine transformation
$ (A,B)\mapsto(uA+v,uB+v)$, where $u\in\F^*$ and $v\in\F$, because
$(uA+v)\rs(uB+v)=u(A\rs B)+2v$.  Independent translations, however, need not preserve the excluded
diagonal.  We use
\emph{common-endpoint progressions} to mean a common affine image of
$([0,a-1],[0,b-1])$; reversing the common difference includes the
terminal-segment description.  A set $S\subseteq\F$ is
\emph{centrally symmetric} if $S=2h-S$ for some $h\in\F$.

\subsection{Prior results and the remaining problem}

The closest precursor to the present work is the paper of Liu and Qian
\cite{LQ}, which treats the inverse problem under a progression
hypothesis.  In particular, if $b\ge3$, $a\ge b+3$, and $a+b\le p$,
and at least one summand is an arithmetic progression, their
Theorem~1.5 forces a critical pair to consist of progressions with the
same difference and a common endpoint.  Their Conjecture~1.7 asks whether
the progression hypothesis can be removed in this range.
For $b\ge4$ and $a-b\in\{1,2\}$, their Proposition~3.2(i) determines
the deletion patterns when $A$ is an arithmetic progression, whereas
their Proposition~3.4 gives the same necessary forms under the weaker
hypothesis $B\subsetneq A$; both apply when
$|A\rs B|=a+b-2\le p-2$.
Thus removing the progression hypothesis in the large gap and deriving
containment in the two small gaps are genuinely different tasks.

Subsequent examples show that the progression-only conjecture requires
two corrections.  First, the boundary $a+b=p$ already contains infinite
families of non-progression critical pairs, constructed by Li and Liang
\cite[Theorems 2.1 and 2.3]{LL}.  Li, Li, and Yuan
\cite[Theorem 1.1]{HLY} subsequently classified this entire boundary
by explicit two-missing-sum models.  Every model has the smaller set
contained in the larger \cite[Proposition 3.2(i)]{HLY}.
Second, a non-progression obstruction remains strictly inside the boundary:
Li, Li, and Yuan \cite[Example 4.1]{HLY} exhibited a critical pair in
$\mathbb F_{13}$ with $(a,b)=(7,4)$ and $a+b=p-2$.  This pair also satisfies
the size assumptions of the progression-only assertion in
\cite[Theorem 3.1, version 2]{LL}, so that assertion does not hold as
stated.  The natural question below the boundary is therefore not whether
all large-gap pairs are progressions, but whether this $\mathbb F_{13}$
orbit is the only exception.

The equal-cardinality case is a separate ingredient needed for a global
classification.  Daza Urbano, Gonz\'alez-Mart\'inez, Huicochea Mason, and
Montejano Cantoral \cite[Theorem 7]{DGHM} showed that for
$|A|=|B|=k\ge5$ and $p>2k-1$, the equality
$|A\rs B|=2k-2$ forces $A=B$.  Thus distinct sets of these sizes cannot be
critical in the interior.  The remaining small cardinalities are treated in
Section~\ref{classification}, after the unequal-cardinality cases.

For $b\ge3$, the unresolved interior problem can now be stated sharply.
When $a-b\ge3$, one must determine whether the known $\mathbb F_{13}$
orbit is the unique failure of progression rigidity.  When
$a-b\in\{1,2\}$, one must prove rather than assume the containment
$B\subsetneq A$.  Theorem~\ref{main} resolves both questions and combines
them with the equal-cardinality, small-cardinality, boundary, and
saturated cases into a single inverse theorem.

\subsection{Complete classification}

The number of missing restricted sums provides a common parameter for
the different parts of the answer.  Define $m:=|\F\setminus(A\rs B)|$.
For a critical pair, $m=\max\{0,p-a-b+2\}$.  Thus $m=0$ is the saturated case, $m=1,2$ correspond to $a+b=p+1,p$,
and $m\ge3$ corresponds to the \emph{interior} $a+b\le p-1$.
These correspondences are used only for critical pairs.  Within the
interior, the \emph{large-gap range} means $b\ge3$ and $a-b\ge3$;
the gaps $a-b=1,2$ and the small values of $b$ require different
structural descriptions.

We first record the normal forms that appear in the statement.  Here and
below, $[r,s]$ denotes the residues of the consecutive integers
$r,r+1,\ldots,s$, always with fewer than $p$ terms; an empty index range
means the empty set.  Define
\begin{align}
\mathcal P_h(B):\quad &h\in B,\quad
 A=(\F\setminus(2h-B))\cup\{h\};\label{plusform}\\
 \mathcal T_{r,s}:\quad&A=\F\setminus K_{r,s},\quad B=B_{r,s},\label{twoform}\\
 \mathcal E:\quad&p=13,\quad A_* =\{0,1,3,4,5,6,10\},\quad
 B_* =\{0,1,4,6\};\label{exception}
\end{align}
where
\begin{align}
B_{r,s}&=\{-2i:0\le i<r\}\cup\{2j+1:0\le j<s\},\notag \\
K_{r,s}&=\{2i:1\le i\le r\}\cup\{-(2j+1):0\le j<s\}.\notag
\end{align}
In \eqref{plusform} we retain only $A\ne B$ and $a\ge b$.  In
\eqref{twoform}, $r,s\ge0$ and $r+s=b\le(p-1)/2$.  In
\eqref{exception}, $\mathcal E$ denotes the common affine orbit
represented by $(A_*,B_*)$.

\begin{theorem}\label{main}
Let $p$ be an odd prime and let $A,B\subseteq\F$ be nonempty distinct
sets, with $a=|A|\ge b=|B|$. Then $(A,B)$ is critical if and only if,
up to common affine equivalence, it is one of the pairs listed in
Table~\ref{total}.
\end{theorem}

\begingroup
\renewcommand{\arraystretch}{1.08}
\begin{longtable}{@{}>{\raggedright\arraybackslash}p{.20\textwidth}>{\raggedright\arraybackslash}p{.74\textwidth}@{}}
\caption{Complete classification of critical pairs.}
\label{total}\\
\toprule
\textnormal{Parameter range}&\textnormal{Normal form of $(A,B)$}\\
\midrule
\endfirsthead
\toprule
\textnormal{Parameter range}&\textnormal{Normal form of $(A,B)$}\\
\midrule
\endhead
$a+b\ge p+2$&$(A,B)$ arbitrary, with $A\ne B$ and $a\ge b$.\\[3pt]
$a+b=p+1$&$A=(\F\setminus(2h-B))\cup\{h\}$, where $h\in B$.\\[3pt]
$a+b=p$&$A=\F\setminus K_{r,s}$, $B=B_{r,s}$, where $r,s\ge0$ and $r+s=b$.\\
\midrule
\multicolumn{2}{@{}l}{\emph{Interior range $a+b\le p-1$:}}\\[2pt]
$a=b$&No critical pair.\\[3pt]
$b=1$, $a>b$&$B=\{v\}$ with $v\in A$.\\[4pt]
$b=2$, $a>b$&$B=\{0,1\}$, and either\newline
$A=[0,a-1]$, $A=[2-a,1]$, or\newline
$A=\{0,1\}\cup[\ell,\ell+a-3]$, where $3\le\ell$ and $\ell+a-3\le p-2$.\\[4pt]
$b\ge3$, $a-b=1$&$b\ge4$: $A=[0,b]$, $B=A\setminus\{r\}$, $0\le r\le b$;\newline
$b=3$: $A=2h-A$, $B=A\setminus\{v\}$, $h\in\F$, $v\in A$.\\[4pt]
$b\ge3$, $a-b=2$&$A=[0,b+1]$, $B=A\setminus T$, where\newline
$T\in\bigl\{\{0,1\},\{0,2\},\{b-1,b+1\},\{b,b+1\}\bigr\}$.\\[4pt]
$b\ge3$, $a-b\ge3$&$A=[0,a-1]$, $B=[0,b-1]$;\newline
or, when $(p,a,b)=(13,7,4)$, $A=A_*$ and $B=B_*$.\\
\bottomrule
\end{longtable}
\endgroup

The principal new content of Theorem~\ref{main} lies in the interior.
For $b\ge3$ and $a-b\ge3$, no progression or containment hypothesis is
assumed: the common-endpoint progression family and the single orbit
$\mathcal E$ exhaust the possibilities.  The existence of $\mathcal E$
is due to \cite[Example 4.1]{HLY}; the assertion here is its uniqueness
throughout the large-gap interior.  For $a-b\in\{1,2\}$, containment
$B\subsetneq A$ is derived from criticality, rather than imposed.
For $b\ge4$, the necessary deletion patterns then follow from
\cite[Proposition 3.4]{LQ}; their sufficiency is checked directly below.
We also retain the complete four-element centrally symmetric family
when $(a,b)=(4,3)$.  Liu and Qian \cite[Remark 3.5]{LQ} already gave
a non-progression example at this endpoint.

The boundary row $a+b=p$ is the complete classification of
\cite[Theorem 1.1]{HLY}, rewritten in our normalization.  Likewise, once
containment has been derived, the small-gap deletion patterns for $b\ge4$
recover those of Liu and Qian \cite[Propositions 3.2(i) and 3.4]{LQ}.
The parameter ranges in Table~\ref{total} are disjoint, although
different normal forms within a row may represent the same affine orbit.

\subsection{Proof strategy}

The starting point is the information carried by missing restricted sums.
For a critical pair with $m>0$, write $A^c=\F\setminus A$ and let $H$
consist of the halves of the missing sums.  Then
\[
 x\in B,\ h\in H,\ x\ne h\quad\Longrightarrow\quad2h-x\in A^c.
\]
Subsection~\ref{reflectionclosure} records this pointwise reflection
condition and encodes it in a single identity involving the
root polynomials of $B$, $A^c$, and $H$.  Together with the cardinality
relation, this identity also yields a critical pair.

The coefficient analysis is carried out in the large-gap interior.
Its use of polynomial identities and low-order coefficients is related
to K\'arolyi's approach to the restricted self-sum inverse problem
\cite[Section 2]{Karolyi}.  Here the reflection identity involves three
root polynomials and leads to a mixed coefficient system.
Subsection~\ref{recurrence} derives the levelwise matrix recursion
\[
 \mathsf M_j\mathbf x_j=\mathbf v_j,
 \qquad
 \mathbf v_j=\mathsf Q_j\mathbf h_j+(-1)^{j+1}\mathsf R_j\mathbf c_j,
\]
solves the first three levels, and proves rank three for
$4\le j\le b$.  The first three levels divide the initial data into
three branches; unique continuation and support propagation then
control their later states.  Subsection~\ref{eliminationdetails} tests
these branches by matrix compatibility, together with the actual-degree
conditions and distinctness of the roots.  In the range $4\le j\le b$,
the linear compatibility test is
$\mathbf v_j\in\operatorname{im}\mathsf M_j$; beyond degree $b$, the
coefficient comparisons are obtained directly from the two general
matrix lemmas.  The fourth, fifth, and sixth levels eliminate every
non-reference branch except the quartic exception in characteristic
$13$.  Proposition~\ref{coefficientrigidity} then uses
unique continuation to identify all coefficients through degree $b$ in
the reference branch.  Appendix~\ref{app:finitechecks} records the
detailed finite compatibility calculations omitted from the main text.

Section~\ref{classification} first recovers the large-gap pairs from
Proposition~\ref{coefficientrigidity}.  The inverse theorem of Liu and
Qian \cite[Theorem 1.5(ii)]{LQ} determines $A$ when $B$ is an arithmetic progression;
in the exceptional case, the reflection constraints determine $A$ directly.  Subsection~\ref{smallgapinterior} settles $b=1,2$
and then treats the two small gaps.  Its key step is containment:
when $A\cap B\ne\varnothing$, criticality ensures that compression
to $(A\cup B,A\cap B)$ preserves the restricted sumset,
and missing-sum reflections rule out a nontrivial partition of the
compressed difference set.  Thus criticality forces $B\subsetneq A$;
the known deletion theorem and the two low-cardinality endpoints finish
the small-gap classification.  Subsection~\ref{remainingcases} treats
equal cardinalities, the two boundary layers, and saturation, citing
\cite{DGHM,HLY} where they apply and supplying the remaining short proofs.

\section{Structural and coefficient framework}\label{structure}

This section develops the coefficient argument for the large-gap interior.
Subsection~\ref{reflectionclosure} starts from the elementary reflection
implication forced by each missing restricted sum, encodes that implication
by a polynomial identity, and records the converse reconstruction criterion.
In Subsection~\ref{recurrence} we impose the large-gap assumptions,
derive the explicit matrix recursion
$\mathsf M_j\mathbf x_j=\mathbf v_j$, solve the first three levels, and
prove the stable rank, unique-continuation, and support-propagation
properties from level four onward.  These calculations isolate three
low-order branches.  Subsection~\ref{eliminationdetails} tests them by
the compatibility condition $\mathbf v_j\in\operatorname{im}\mathsf M_j$;
only levels four, five, and six, together with the actual-degree endpoint
equations, are needed.  The resulting coefficient-rigidity proposition
identifies the target polynomial $F_B$, so that
Subsection~\ref{largegapinterior} can recover the underlying sets without
repeating the coefficient analysis.

Throughout Sections~\ref{structure}--\ref{classification} and
Appendix~\ref{app:finitechecks}, integer parameters in field expressions
are identified with their images in $\F$; inequalities involving these
parameters refer to their integer values.  Column vectors are written in
bold, and matrix names in sans serif.  For $k\ge0$ we write
\[
 \fall{x}{k}=x(x-1)\cdots(x-k+1),\qquad
 \rise{x}{k}=x(x+1)\cdots(x+k-1),
\]
with both products equal to $1$ when $k=0$.  The symbol $\ind_E$ denotes
the indicator of a condition $E$.

\subsection{Reflection encoding and matrix tools}\label{reflectionclosure}

For $h\in\F$, the map $x\mapsto2h-x$ is the \emph{reflection about $h$}.
Since $p$ is odd, its unique fixed point is $h$, called the
\emph{reflection center}.
Let $(A,B)$ be a critical pair with $a\ge b$ and $m>0$, put
$A^c:=\F\setminus A$, and define $H:=(\F\setminus(A\rs B))/2$.
Thus $|H|=m$ and $h\in H$ precisely when $2h\notin A\rs B$.
For $x\in B$ and $h\in H$, one has $x\ne h\Rightarrow2h-x\in A^c$;
otherwise $2h=(2h-x)+x$ would be an admissible restricted representation.
The exception $x=h$ corresponds exactly to the excluded diagonal pair.
This pointwise reflection implication is the only set-theoretic input
needed for the polynomial encoding.  We therefore isolate it from the
critical-pair setting and encode it for arbitrary sets.  In particular,
the next lemma requires neither criticality nor the cardinality relation
among $A$, $B$, and $H$.

For nonempty $A,B\subseteq\F$ and any $H\subseteq\F$, define
\begin{equation}\label{rootpolys-set}
 F_B(X):=\prod_{x\in B}(X-x),\qquad
 F_{A^c}(Z):=\prod_{y\in A^c}(Z-y),\qquad
 F_H(Y):=\prod_{h\in H}(Y-h),
\end{equation}
with an empty product interpreted as $1$, and put
\begin{equation}\label{reflectionpoly}
 P(X,Y):=(X-Y)F_{A^c}(2Y-X).
\end{equation}
The factor $X-Y$ encodes the allowed diagonal exception, so the
reflection condition is equivalent to the vanishing of $P$ on $B\times H$.
For $H\ne\varnothing$, the ideal-membership statement below is the
two-variable Cartesian-grid case of Alon's Combinatorial Nullstellensatz
\cite[Theorem 1.1]{AlonCN}.  We include the short division proof; if
$H=\varnothing$, the identity is immediate from $F_H=1$.

\begin{lemma}\label{gridlemma}
Let $A,B\subseteq\F$ be nonempty and $H\subseteq\F$, with root
polynomials as in \eqref{rootpolys-set}.  The condition
\[
 2h-x\in A^c\qquad(x\in B,\ h\in H,\ x\ne h)
\]
holds if and only if there exist $U,V\in\F[X,Y]$ such that
\begin{equation}\label{witnessidentity}
 P(X,Y)=F_B(X)U(X,Y)+F_H(Y)V(X,Y).
\end{equation}
\end{lemma}
\begin{proof}
Assume the reflection condition.  The factor $X-Y$ gives vanishing when $x=h$; in all other cases,
vanishing follows from $F_{A^c}(2h-x)=0$.  Hence $P=0$ on
$B\times H$.  Divide by the monic $F_B$ in $\F[Y][X]$ to write
$P=F_BQ+G$ with $\deg_XG<b$.  For every $h\in H$, the polynomial
$G(X,h)$ vanishes at all $b$ elements of $B$, so it is identically zero.
Writing $G=\sum_{i=0}^{b-1}g_i(Y)X^i$, we deduce that each $g_i$
vanishes on $H$ and is divisible by $F_H$.  Thus $G=F_HV$, giving
\eqref{witnessidentity} with $U=Q$.
Conversely, evaluating that identity at $(x,h)\in B\times H$ gives
$(x-h)F_{A^c}(2h-x)=0$.  If $x\ne h$, its first factor is nonzero,
so $2h-x\in A^c$.
\end{proof}

The identity becomes a characterization of critical pairs once the
cardinality relation is imposed.  This converse will allow us to apply
set-theoretic conclusions to split polynomial solutions.

\begin{lemma}\label{reconstruction}
Let $A,B,H\subseteq\F$ be nonempty, with $A\ne B$, and suppose
\[
 |A^c|=|B|+|H|-2.
\]
If the root polynomials satisfy \eqref{witnessidentity}, then
\[
 |A\rs B|=|A|+|B|-2,\qquad 2H=\F\setminus(A\rs B).
\]
\end{lemma}
\begin{proof}
Lemma~\ref{gridlemma} gives the reflection condition.  If
$2h=a_0+x$ were an admissible representation with $h\in H$, then
$a_0\ne x$ would imply $x\ne h$, and hence
$a_0=2h-x\in A^c$, a contradiction.  Therefore
$2H\subseteq\F\setminus(A\rs B)$, and the cardinality relation gives
$|A\rs B|\le p-|H|=|A|+|B|-2<p$.  Since $A\ne B$, \eqref{anr} supplies the reverse inequality.
Equality follows, and the two sets $2H$ and $\F\setminus(A\rs B)$
have the same size because multiplication by $2$ is bijective.
\end{proof}

The preceding two lemmas complete the passage between the reflection
constraints and the polynomial witness identity.  Before turning to its
coefficient structure, we record two elementary matrix facts that will be
used repeatedly: one for reduction modulo a monic polynomial, and one for
monic polynomial division.

\smallskip
Let $K$ be a field, let $b\ge1$ be an integer, and let
$F(X)=X^b+a_1X^{b-1}+\cdots+a_b\in K[X]$ be monic.  For $G\in K[X]$, write $\operatorname{rem}_{F}(G)$ for the
unique remainder when $G$ is divided by $F$; thus
\[
 G=FQ+\operatorname{rem}_{F}(G),\qquad
 \deg \operatorname{rem}_{F}(G)<b
\]
for a unique $Q\in K[X]$, with $\deg 0=-\infty$.  We also use
$[X^k]G$ to denote the
coefficient of $X^k$ in $G$.

With respect to the ordered basis $(1,X,\ldots,X^{b-1})$ of
$K[X]/(F)$, let $\mathbf e_k\in K^b$ be the coordinate vector of $X^k$
for $0\le k<b$, and set
\begin{equation*}
 \mathsf C_F:=
 \begin{pmatrix}
  0&0&\cdots&0&-a_b\\
  1&0&\cdots&0&-a_{b-1}\\
  0&1&\ddots&\vdots&-a_{b-2}\\
  \vdots&\ddots&\ddots&0&\vdots\\
  0&\cdots&0&1&-a_1
 \end{pmatrix}.
\end{equation*}
For $b=1$, this means $\mathsf C_F=(-a_1)$.  Indeed,
\[
 \mathsf C_F\mathbf e_k=\mathbf e_{k+1}\quad(0\le k<b-1),
 \qquad
 \mathsf C_F\mathbf e_{b-1}
 =-a_b\mathbf e_0-a_{b-1}\mathbf e_1-\cdots-a_1\mathbf e_{b-1},
\]
because $X^b\equiv-a_1X^{b-1}-\cdots-a_b\pmod F$.
Thus $\mathsf C_F$ is precisely the matrix of multiplication by $X$
in $K[X]/(F)$.  In particular, powers of $\mathsf C_F$ encode repeated
multiplication by $X$ followed by reduction modulo $F$.  The following
lemma records this observation in coefficient form.

\begin{lemma}\label{companionremainder}
With the notation above, for $0\le i,k<b$ and $\ell\ge0$,
\begin{equation}\label{general_remainder_matrix}
 [X^k]\operatorname{rem}_{F}(X^{i+\ell})
 =\mathbf e_k^{\mathsf T}\mathsf C_F^\ell \mathbf e_i.
\end{equation}
\end{lemma}
\begin{proof}
The coordinate vector of $X^i$ in $K[X]/(F)$ is $\mathbf e_i$.  Since
multiplication by $X$ is represented by $\mathsf C_F$, the coordinate
vector of the residue class of $X^{i+\ell}$ is
$\mathsf C_F^\ell \mathbf e_i$.  The coordinate indexed by $k$ is therefore the coefficient of $X^k$ in $\operatorname{rem}_{F}(X^{i+\ell})$, which is
exactly \eqref{general_remainder_matrix}.
\end{proof}

\begin{lemma}\label{toeplitzdivision}
Let $K$ be a field, let $m\ge1$ and $d\ge0$ be integers, and let
\[
 F(X)=X^m+f_1X^{m-1}+\cdots+f_m,
 \qquad
 G(X)=X^{m+d}+g_1X^{m+d-1}+\cdots+g_{m+d}
\]
be monic polynomials in $K[X]$.  Then
\begin{equation}\label{general_divisibility_criterion}
 F\mid G
 \quad\Longleftrightarrow\quad
 \mathbf g^{\mathrm{tail}}
 =\mathsf B_{F,d}\mathsf H_{F,d}^{-1}\mathbf g^{\mathrm{top}},
\end{equation}
where
\[
 \mathbf g^{\mathrm{top}}=(1,g_1,\ldots,g_d)^{\mathsf T},
 \qquad
 \mathbf g^{\mathrm{tail}}=(g_{d+1},\ldots,g_{d+m})^{\mathsf T},
\]
and
\[
 \mathsf H_{F,d}:=
 \begin{pmatrix}
  1&0&0&\cdots&0\\
  f_1&1&0&\ddots&\vdots\\
  f_2&f_1&1&\ddots&0\\
  \vdots&\vdots&\ddots&\ddots&\vdots\\
  f_d&f_{d-1}&\cdots&f_1&1
 \end{pmatrix},
 \qquad
 \mathsf B_{F,d}:=
 \begin{pmatrix}
  f_{d+1}&f_d&\cdots&f_1\\
  f_{d+2}&f_{d+1}&\cdots&f_2\\
  \vdots&\vdots&\ddots&\vdots\\
  f_{d+m}&f_{d+m-1}&\cdots&f_m
 \end{pmatrix},
\]
with $f_0=1$ and $f_t=0$ for $t<0$ or $t>m$.
When $d=0$, these definitions give
$\mathsf H_{F,0}=(1)$, $\mathsf B_{F,0}=(f_1,\ldots,f_m)^{\mathsf T}$,
and $\mathbf g^{\mathrm{top}}=(1)$.
\end{lemma}
\begin{proof}
Let
\[
 Q(X)=X^d+q_1X^{d-1}+\cdots+q_d,
 \qquad
 \mathbf q=(1,q_1,\ldots,q_d)^{\mathsf T}.
\]
The coefficient identity $G=FQ$ is equivalent to
\[
 \begin{pmatrix}
  \mathbf g^{\mathrm{top}}\\[1mm]
  \mathbf g^{\mathrm{tail}}
 \end{pmatrix}
 =
 \begin{pmatrix}
  \mathsf H_{F,d}\\[1mm]
  \mathsf B_{F,d}
 \end{pmatrix}
 \mathbf q.
\]
The top block gives
$\mathbf q=\mathsf H_{F,d}^{-1}\mathbf g^{\mathrm{top}}$.
Substituting this into the bottom block gives
\eqref{general_divisibility_criterion}.  Conversely, if
\eqref{general_divisibility_criterion} holds, define
$\mathbf q=\mathsf H_{F,d}^{-1}\mathbf g^{\mathrm{top}}$; then both
blocks of the coefficient identity hold, so $G=FQ$.
\end{proof}

We now restrict the coefficient analysis to the large-gap interior.
The witness identity involves all three root polynomials; its coefficient
recursion follows by applying Lemmas~\ref{companionremainder} and
\ref{toeplitzdivision} to reduction modulo $F_B$ and division by $F_H$,
respectively.

\subsection{Coefficient recursion and unique continuation}
\label{recurrence}

We now apply the two matrix lemmas above to the witness identity of
Lemma~\ref{gridlemma}.  At each level, we regard the preceding coefficients as given.
After introducing the notation, we derive the recursion in
Proposition~\ref{largegaprecursion} by a coefficient calculation for $P$
and one application of each general matrix lemma.  We subsequently resolve
the three initial levels and prove the stable rank and continuation
properties from level four onward.  Subsection~\ref{eliminationdetails}
eliminates the incompatible low-order branches.

Throughout these two subsections, we work in the large-gap interior
$b\ge3$, $a\ge b+3$, and $a+b\le p-1$.  Since $m=p-a-b+2$, these
conditions are equivalent to
\begin{equation}\label{rangealllarge}
 b\ge3,\qquad m=p-a-b+2,\qquad
 3\le m\le p-2b-1.
\end{equation}
Write $n=p-a=m+b-2$, the degree of $F_{A^c}$.
We use the same degree assumptions for split solutions of
\eqref{witnessidentity} with distinct roots; in that setting,
Lemma~\ref{reconstruction} recovers the corresponding critical pair.

We first remove the common translation freedom.  A common
translation of $A$ and $B$ preserves criticality and translates
$A^c$ and $H$ by the same amount.  Since $b<p$, we may translate by
$-b^{-1}\sum_{x\in B}x$ and assume $\sum_{x\in B}x=0$.  We keep the
same notation for the translated sets.
This normalization removes the coefficient of $X^{b-1}$ from $F_B$.
Write
\begin{equation}\label{rootpolys}
\begin{aligned}
 F_B(X)&=X^b+\sum_{s=2}^{b}u_sX^{b-s},\\
 F_{A^c}(Z)&=\sum_{s=0}^{n}(-2)^sc_sZ^{n-s},\\
 F_H(Y)&=\sum_{s=0}^{m}h_sY^{m-s}.
\end{aligned}
\end{equation}
The factors $(-2)^s$ are chosen for the substitution
$2Y-X=-2(X/2-Y)$ in \eqref{reflectionpoly}.
Here $c_0=h_0=u_0=1$ and $u_1=0$.

The index $s$ records the drop from the leading degree in each
polynomial.  We therefore collect the three coefficients at level $s$ into
$\mathbf x_s=(c_s,h_s,u_s)^{\mathsf T}$ for $s\ge0$, and assign
$\operatorname{wt}(c_s)=\operatorname{wt}(h_s)=\operatorname{wt}(u_s)=s$.
Accordingly, the state $\mathbf x_s$ has weight $s$.  The initial state
$\mathbf x_0=(1,1,1)^{\mathsf T}$ is fixed by monicity and has weight zero.  Coefficients beyond the
actual degrees are always taken to be zero:
\begin{equation}\label{actual_degree_conditions}
 u_j=0\quad(j>b),\qquad
 c_j=0\quad(j>n),\qquad
 h_j=0\quad(j>m).
\end{equation}

Fix a level $1\le j\le b$, and regard the preceding states
$\mathbf x_0,\ldots,\mathbf x_{j-1}$ as specified.  The prospective
state $\mathbf x_j$ is subject to the same centering and degree
conventions; no extension satisfying the witness identity is assumed at
this stage.  Set
\begin{equation*}
 \mathbf h_j=(h_{j-1},\ldots,h_1)^{\mathsf T},\qquad
 \mathbf c_j=(c_{j-1},\ldots,c_1,c_0)^{\mathsf T}.
\end{equation*}
Thus $\mathbf h_j\in\F^{j-1}$ and $\mathbf c_j\in\F^j$;
$\mathbf h_1$ is the empty vector.  The basis vectors have indices
$0,\ldots,b-1$, whereas the rows of the level-$j$ system have indices
$1,\ldots,j$, related by $i=b-r$.

For $0\le i<b$ and $0\le\ell\le k\le b$, define
\begin{equation}\label{lambda}
 \lambda_i(k,\ell):=
 \frac{\fall{n-i+1}{k}(n-k+i+2\ell+1)}
 {2^\ell\fall n{k-\ell}\rise{i+1}{\ell}(n+i+1)}.
\end{equation}
All denominators are units under \eqref{rangealllarge}.

The preceding $u$-coefficients determine
\begin{equation*}
 F_B^{<j}(X):=X^b+\sum_{2\le q<j}u_qX^{b-q},
 \qquad
 \mathsf C_j:=\mathsf C_{F_B^{<j}},
\end{equation*}
where the companion matrix on the right is the one defined before
Lemma~\ref{companionremainder}.  Define the $j\times j$ matrix
\begin{equation*}
 (\mathsf R_j)_{r,\ell}:=
 \lambda_{b-r}(j,\ell)\,
 \mathbf e_{b-r}^{\mathsf T}\mathsf C_j^\ell \mathbf e_{b-r},
 \qquad 1\le r,\ell\le j.
\end{equation*}

For $1\le r\le j$, put $i=b-r$ and define the $r\times r$ matrix
\begin{equation*}
 \mathsf H_r:=
 \begin{pmatrix}
 1&0&\cdots&0\\
 h_1&1&\ddots&\vdots\\
 \vdots&\ddots&\ddots&0\\
 h_{r-1}&\cdots&h_1&1
 \end{pmatrix}.
\end{equation*}
This is $\mathsf H_{F_H,r-1}$ in Lemma~\ref{toeplitzdivision}.  Define
\begin{equation}\label{history_quotient_vector}
 \widehat{\mathbf q}_r:=
 \mathsf H_r^{-1}
 \operatorname{diag}\!\bigl(
  1,-\lambda_i(1,0),\ldots,(-1)^{r-1}\lambda_i(r-1,0)
 \bigr)
 \begin{pmatrix}c_0\\c_1\\\vdots\\c_{r-1}\end{pmatrix}
 =(1,\widehat q_{r,1},\ldots,\widehat q_{r,r-1})^{\mathsf T}.
\end{equation}
Thus $\widehat{\mathbf q}_r\in\F^r$, with
$\widehat{\mathbf q}_1=(1)$.  Assemble these coefficients into the
$j\times(j-1)$ matrix
\begin{equation*}
 \mathsf Q_j:=
 \begin{pmatrix}
 0&0&\cdots&0\\
 \widehat q_{2,1}&0&\cdots&0\\
 \widehat q_{3,1}&\widehat q_{3,2}&\ddots&\vdots\\
 \vdots&\vdots&\ddots&0\\
 \widehat q_{j,1}&\widehat q_{j,2}&\cdots&\widehat q_{j,j-1}
 \end{pmatrix},
\end{equation*}
with $\mathsf Q_1$ the unique $1\times0$ matrix.

Finally, define the $j\times3$ coefficient matrix
\begin{equation}\label{row}
 \mathsf M_j=
 \begin{pmatrix}
 (-1)^j\lambda_{b-1}(j,0)&-1&(-1)^{j+1}\lambda_{b-1}(j,j)\\
 (-1)^j\lambda_{b-2}(j,0)&-1&(-1)^{j+1}\lambda_{b-2}(j,j)\\
 \vdots&\vdots&\vdots\\
 (-1)^j\lambda_{b-j}(j,0)&-1&(-1)^{j+1}\lambda_{b-j}(j,j)
 \end{pmatrix}.
\end{equation}

\begin{proposition}\label{largegaprecursion}
Assume \eqref{rangealllarge} and \eqref{witnessidentity}.  For every
$1\le j\le b$,
\begin{equation}\label{coefficient_system}
 \mathsf M_j\mathbf x_j=\mathbf v_j,
 \qquad
 \mathbf v_j:=\mathsf Q_j\mathbf h_j
              +(-1)^{j+1}\mathsf R_j\mathbf c_j,
\end{equation}
where every entry of $\mathbf v_j$ is weighted-homogeneous of weight
$j$ in the preceding scalar coefficients.
\end{proposition}
\begin{proof}
Write
\[
 -2^{-n}P(X,Y)\equiv\sum_{i=0}^{b-1}G_i(Y)X^i
                  \pmod{F_B(X)}.
\]
By Lemma~\ref{gridlemma}, \eqref{witnessidentity} is equivalent to
$F_H\mid G_i$ for every $0\le i<b$.

Fix $1\le r\le j$ and put $i=b-r$.  Reduction modulo $F_B$ cannot
increase the total degree in $X,Y$, so $\deg G_i\le n+1-i=m+r-1$.
Write
\[
 G_i(Y)=\sum_{s=0}^{m+r-1}g_{i,s}Y^{m+r-1-s},
 \qquad g_{i,s}=0\quad(s>m+r-1).
\]
In the expansion of $-2^{-n}P$, the summand
$-(-1)^sc_s(X-Y)(Y-X/2)^{n-s}$ has total degree $n+1-s$ in $X,Y$.
Every nontrivial reduction step replaces a power $X^b$ by
$-\sum_{q=2}^{b}u_qX^{b-q}$ and strictly lowers this total degree.
Consequently the coefficient of $X^iY^{n+1-i}$ can only come from
the unreduced summand with $s=0$.  The direct coefficient identity
\[
 [X^qY^{\nu+1-q}]\bigl(-(X-Y)(Y-X/2)^\nu\bigr)
 =(-1)^q2^{-q}
 \left(\binom{\nu}{q}+2\binom{\nu}{q-1}\right)
\]
gives
\begin{equation}\label{leading_gi}
 g_{i,0}=(-1)^i2^{-i}
 \left(\binom ni+2\binom n{i-1}\right)
 =\begin{cases}
 1,&i=0,\\[1mm]
 (-1)^i\dfrac{\fall n{i-1}(n+i+1)}{2^ii!},&i\ge1,
 \end{cases}
\end{equation}
Thus $g_{i,0}\ne0$, since every factor in the displayed expression
is a unit under \eqref{rangealllarge}.  For $0\le\ell\le j$ and
$j\le n+1-i$, the term with
$s=j-\ell$ contributes to $X^{i+\ell}Y^{n+1-i-j}$ before reduction.
Removing the explicit sign factors gives the coefficient ratio
\begin{equation*}
 \frac{2^{-(i+\ell)}
  \left(\binom{n-j+\ell}{i+\ell}
              +2\binom{n-j+\ell}{i+\ell-1}\right)}
 {2^{-i}\left(\binom ni+2\binom n{i-1}\right)}
 =\lambda_i(j,\ell).
\end{equation*}
Here out-of-range binomial coefficients are zero.  In the nonzero
range, expanding the factorials and cancelling gives the expression
in \eqref{lambda}; when $j>n+1-i$, both sides are zero.
The sign before normalization is $(-1)^{s+i+\ell}=(-1)^{i+j}$.
After division by $g_{i,0}$, the remaining sign is $(-1)^j$, giving the
normalized contribution $(-1)^j\lambda_i(j,\ell)c_{j-\ell}$.

Apply Lemma~\ref{companionremainder} to $F_B^{<j}$.  By the definition
of $\mathsf R_j$, all contributions involving only
$u_2,\ldots,u_{j-1}$ are collected in
$(-1)^j(\mathsf R_j\mathbf c_j)_r$.  A use of $u_q$ lowers the total
degree in $X,Y$ by $q$.  Thus a contribution of weight $j$ involving
$u_j$ must start from $c_0=1$ and use $u_j$ exactly once, with no other
nonconstant coefficient.  This gives precisely
\[
 X^{i+j}=X^{i+j-b}X^b\longmapsto-u_jX^i.
\]
Therefore
\begin{equation}\label{matrix_normalized_gij}
 \frac{g_{i,j}}{g_{i,0}}
 =(-1)^j\lambda_i(j,0)c_j
  +(-1)^{j+1}\lambda_i(j,j)u_j
  +(-1)^j(\mathsf R_j\mathbf c_j)_r.
\end{equation}

For $1\le s<r$ we have $i+s<b$, so no reduction occurs.  Hence the
first $r$ coefficients of the monic polynomial $G_i/g_{i,0}$ are
exactly
\[
 \operatorname{diag}\!\bigl(
  1,-\lambda_i(1,0),\ldots,(-1)^{r-1}\lambda_i(r-1,0)
 \bigr)
 \begin{pmatrix}c_0\\c_1\\\vdots\\c_{r-1}\end{pmatrix}.
\]
Apply Lemma~\ref{toeplitzdivision} with $F=F_H$,
$G=G_i/g_{i,0}$, and $d=r-1$.  The quotient coefficient vector is
$\widehat{\mathbf q}_r$.  If $j\le m+r-1$, the coefficient identity at level $j$ is
\begin{equation*}
 \frac{g_{b-r,j}}{g_{b-r,0}}-h_j
 =\sum_{s=1}^{r-1}\widehat q_{r,s}h_{j-s}
 =(\mathsf Q_j\mathbf h_j)_r.
\end{equation*}
Combining this with \eqref{matrix_normalized_gij} gives the $r$-th row
of \eqref{coefficient_system} whenever $j\le m+r-1$.  If $j>m+r-1$, then
$\lambda_{b-r}(j,\ell)=0$ for every $0\le\ell\le j$, while
$h_j=h_{j-s}=0$ for $1\le s<r$; thus the same row equation reduces to
the degree convention $-h_j=0$.  Hence \eqref{coefficient_system}
holds for all $1\le r\le j$.

Finally, the coefficient
$\mathbf e_{b-r}^{\mathsf T}\mathsf C_j^\ell \mathbf e_{b-r}$ is
weighted-homogeneous of weight $\ell$ in the preceding $u$-coefficients,
so every term of $\mathsf R_j\mathbf c_j$ has weight $j$.
The triangular system defining $\widehat{\mathbf q}_r$ similarly gives
$\operatorname{wt}(\widehat q_{r,s})=s$, and therefore every term of
$\mathsf Q_j\mathbf h_j$ also has weight $j$.  This proves the stated
weight assertion.
\end{proof}

The formula defining $\mathbf v_j$ and the weight calculation above
are valid for arbitrary preceding states satisfying the centering and
degree conventions.  The witness identity implies that the current state satisfies
\eqref{coefficient_system}; satisfying a single level does not imply
that the data extend to a full polynomial solution.
Write $V_{r,j}:=(\mathbf v_j)_r$ when individual components are needed.

The first three levels require separate treatment.  We solve them
directly from
\eqref{coefficient_system}, starting with the monic initial state
$\mathbf x_0=(1,1,1)^{\mathsf T}$.  Put
$t=c_1=-\frac12\sum_{x\in A}x$ and $z=m+2b-3=n+b-1$, and
\begin{equation*}
 \Theta_b=\frac{2b(b^2-1)}3,\qquad
 \Delta_2=z^2-2z+2(b-2),\qquad
 \Delta_3=z^2-z+6(b-3).
\end{equation*}
We shall also use
\begin{equation}\label{vartheta3def}
 \vartheta_3=\frac{8(b^2-4)(z-1)z}{n^3(z+1)^3}\in\F^*.
\end{equation}
The displayed denominators are units because $n\ge b+1$, $z\ge6$,
and $z+3=m+2b\le p-1$.

At level $j=1$, one has $F_B^{<1}=X^b$, so $\mathsf R_1=0$,
while $\mathsf Q_1\mathbf h_1=0$.  Hence Proposition~\ref{largegaprecursion}
gives $\mathbf v_1=0$, and $\mathsf M_1\mathbf x_1=0$ is the single equation
$-\lambda_{b-1}(1,0)c_1-h_1+\lambda_{b-1}(1,1)u_1=0$.
The centering condition gives $u_1=0$ and
$\lambda_{b-1}(1,0)=mz/[n(z+1)]$.  Thus $c_1=t$ is the only free first-level scalar and
\begin{equation}\label{h1formula}
 \mathbf x_1=
 \left(t,-\frac{mz}{n(z+1)}t,0\right)^{\mathsf T}.
\end{equation}
In particular, $\mathsf M_1$ has rank one.

For later comparison, when $t\ne0$ define the arithmetic-progression
reference pair
\begin{equation}\label{referenceprogressions}
\begin{aligned}
 \delta&=-\frac{4t}{n(z+1)},\\
 B^0&=\{(k-(b-1)/2)\delta:0\le k<b\},\\
 A^0&=\{(k-(b-1)/2)\delta:0\le k<a\}.
\end{aligned}
\end{equation}
Since $n(z+1)=a(a-b)$ in $\F$, one has
$-\tfrac12\sum_{x\in A^0}x=t$.  The unscaled intervals show directly that
$(A^0,B^0)$ is critical: all sums from $1$ through $a+b-2$ have a
representation with unequal summands, whereas the only representation of $0$ is diagonal; no wraparound
occurs because $a+b\le p-1$.  Let
$H^0=(\F\setminus(A^0\rs B^0))/2$.  Using the normalization in
\eqref{rootpolys}, write
$\mathbf x_s^0=(c_s^0,h_s^0,u_s^0)^{\mathsf T}$ for the coefficient
states of this reference triple.  The superscript $0$ denotes reference
data.  We write $\mathbf v_j^0$ for the right-hand side evaluated on
the reference history.  At $t=0$, these symbols denote the formal data
$\mathbf x_0^0=(1,1,1)^{\mathsf T}$ and $\mathbf x_s^0=0$ for $s>0$.
In particular, $u_2^0=-\Theta_bt^2/[n^2(z+1)^2]$ and $u_3^0=0$.
The remaining reference coefficients through level three are given
in \eqref{reference_low_explicit} in Appendix~\ref{app:loworder}.

At level $j=2$, $F_B^{<2}=X^b$, hence $\mathsf R_2=0$.  Moreover,
\[
 \widehat{\mathbf q}_2=(1,\widehat q_{2,1})^{\mathsf T},\qquad
 \widehat q_{2,1}=-c_1\lambda_{b-2}(1,0)-h_1,
\]
so
\[
 \mathsf Q_2=
 \begin{pmatrix}0\\ \widehat q_{2,1}\end{pmatrix},\qquad
 \mathbf v_2=\mathsf Q_2\mathbf h_2
 =\begin{pmatrix}0\\ h_1\widehat q_{2,1}\end{pmatrix}.
\]
Thus $\mathsf M_2\mathbf x_2=\mathbf v_2$ is the explicit two-equation
system
\begin{equation}\label{second_level_system}
\begin{cases}
 \lambda_{b-1}(2,0)c_2-h_2-\lambda_{b-1}(2,2)u_2=0,\\
 \lambda_{b-2}(2,0)c_2-h_2-\lambda_{b-2}(2,2)u_2
      =h_1\widehat q_{2,1}.
\end{cases}
\end{equation}
The reference state $\mathbf x_2^0$ is a particular solution because
it has the same first state.  Subtracting the second row from the first,
the differences of the $c_2$- and $u_2$-coefficients are
\[
 -\frac{2m\Delta_2}{z(n-1)(z+1)n},\qquad
 \frac{m(n+1)(\Delta_2+6z+6)}{2b(b^2-1)z(z+1)}.
\]
If $\Delta_2\ne0$, the first difference is nonzero; if
$\Delta_2=0$, the second becomes
$3m(n+1)/(b(b^2-1)z)\ne0$.  Thus $\mathsf M_2$ has rank two, so its kernel is
one-dimensional.  Substitution in the homogeneous system associated with
\eqref{second_level_system} shows that its kernel is spanned by
$(\alpha_2,\beta_2,\gamma_2)^{\mathsf T}$, where
\begin{equation}\label{alpha2beta2}
\begin{aligned}
 \alpha_2&=\frac{(n-1)n(n+1)(\Delta_2+6z+6)}
                  {4b(b^2-1)n^2(z+1)^2},\\
 \beta_2&=\frac{m(m^2-1)(\Delta_2+3z)}
                  {4b(b^2-1)n^2(z+1)^2},\qquad
 \gamma_2=\frac{\Delta_2}{n^2(z+1)^2}.
\end{aligned}
\end{equation}
Consequently, every solution at level two has the unique representation
\begin{equation}\label{mregsecond}
 \mathbf x_2=\mathbf x_2^0+
       \sigma(\alpha_2,\beta_2,\gamma_2)^{\mathsf T},
       \qquad \sigma\in\F.
\end{equation}
At level $j=3$, one has
$F_B^{<3}(X)=X^b+u_2X^{b-2}$.  The definitions of $\mathsf R_3$ and
$\mathsf Q_3$ therefore give
\[
 \mathsf R_3=
 \begin{pmatrix}
  0&-u_2\lambda_{b-1}(3,2)&0\\
  0&-u_2\lambda_{b-2}(3,2)&0\\
  0&0&0
 \end{pmatrix},\qquad
 \mathsf Q_3=
 \begin{pmatrix}
  0&0\\
  \widehat q_{2,1}&0\\
  \widehat q_{3,1}&\widehat q_{3,2}
 \end{pmatrix},
\]
where
\[
 \widehat q_{3,1}=-c_1\lambda_{b-3}(1,0)-h_1,
 \qquad
 \widehat q_{3,2}=c_2\lambda_{b-3}(2,0)
                  -h_1\widehat q_{3,1}-h_2.
\]
Proposition~\ref{largegaprecursion} therefore gives
\begin{equation}\label{third_explicit_rows_main}
 \mathbf v_3=
 \begin{pmatrix}
 -c_1u_2\lambda_{b-1}(3,2)\\
 -c_1u_2\lambda_{b-2}(3,2)+h_2\widehat q_{2,1}\\
 h_2\widehat q_{3,1}+h_1\widehat q_{3,2}
 \end{pmatrix}.
\end{equation}
Set
\[
 A_r=-\lambda_{b-r}(3,0),\qquad
 B_r=\lambda_{b-r}(3,3)\qquad(1\le r\le3).
\]
Then the third-level system is
\begin{equation}\label{third_level_system}
 A_rc_3-h_3+B_ru_3=V_{r,3},\qquad r=1,2,3.
\end{equation}
To eliminate $c_3$ and $h_3$, take
$(\eta_1,\eta_2,\eta_3)=(A_3-A_2,A_1-A_3,A_2-A_1)$.  By construction, $\sum_{r=1}^3\eta_r=\sum_{r=1}^3\eta_rA_r=0$.  Substitution of
\eqref{lambda} gives
\[
 \sum_{r=1}^3\eta_rB_r=\det \mathsf M_3=-K_3\Delta_3,
\]
where
\[
 K_3=\frac{9m^2(m^2-1)(n+1)}
 {2b(b^2-1)(b^2-4)z(z-1)n(n-1)(n-2)}\in\F^*.
\]
On the right-hand side, substitute
\eqref{h1formula} and \eqref{mregsecond} into
\eqref{third_explicit_rows_main}.  The reference contribution cancels,
and the remaining $t\sigma$-part satisfies
\[
 \sum_{r=1}^3\eta_rV_{r,3}=K_3\vartheta_3t\sigma.
\]
The finite simplification of this last identity is recorded in
Appendix~\ref{app:loworder}; no division by $\Delta_2$ is used.
Applying the combination to \eqref{third_level_system} and cancelling
the unit $K_3$ yields the necessary and sufficient compatibility
relation
\begin{equation}\label{mregthird}
 \Delta_3u_3+\vartheta_3t\sigma=0.
\end{equation}

The rank of $\mathsf M_3$ follows from these identities.  If $\Delta_3\ne0$, then
$\det \mathsf M_3\ne0$, so the full system uniquely determines the third state, with $u_3$
given by \eqref{mregthird}.  If $\Delta_3=0$, the minor in the
$c_3,h_3$ columns of the first two rows equals
\[
 A_2-A_1=
 -\frac{2m(m-1)(z-1)(z-3)}{z(z+1)n(n-1)(n-2)}\ne0.
\]
Thus $\mathsf M_3$ has rank two.  For each prescribed $u_3$, the first two
rows determine $c_3,h_3$, and the third row is equivalent to
\eqref{mregthird}.  Hence a third state exists exactly when
$t\sigma=0$, and then $u_3$ is the remaining free parameter.

When $t\ne0$, the consistent third-level data can be written with a
single parameter $\xi$:
\begin{equation}\label{mregxi}
 \sigma=\Delta_3\xi,\qquad
 u_3=-\vartheta_3t\xi.
\end{equation}
Indeed, if $\Delta_3\ne0$ take $\xi=\sigma/\Delta_3$; if
$\Delta_3=0$, consistency gives $\sigma=0$ and one takes
$\xi=-u_3/(\vartheta_3t)$.  The choice $\xi=0$ gives the reference
states through order three.  Explicit formulas for $c_3,h_3$ needed
only in the later finite checks are recorded in
Appendix~\ref{app:loworder}, see \eqref{c3h3regular}--\eqref{alpha3beta3}.

Thus $\mathsf M_1$ has rank one, $\mathsf M_2$ has rank two, and
$\mathsf M_3$ has rank three except on $\Delta_3=0$, where it has
rank two.  The next lemma shows that the coefficient matrices have
full column rank throughout the remaining range $4\le j\le b$.

\begin{lemma}\label{rank}
Assume \eqref{rangealllarge}.  For every $4\le j\le b$, the matrix
$\mathsf M_j$ has rank three.
\end{lemma}
\begin{proof}
For the row indexed by $r$, put $x_r=m+r-1$.  Since $n=m+b-2$, the
$r$-th row of $\mathsf M_j$ in \eqref{row} can be written as
$(L_j(x_r),-1,R_j(x_r))$, where
\[
 L_j(x)=(-1)^j\frac{\fall{x}{j}(2n-j+2-x)}
 {\fall{n}{j}(2n+2-x)},\qquad
 R_j(x)=(-1)^{j+1}\frac{\fall{x}{j}(2n+j+2-x)}
 {2^j\rise{n+2-x}{j}(2n+2-x)}.
\]
For any admissible $x$, let $\mathsf N_j(x)$ be the $3\times3$ matrix
whose rows are $(L_j(x+q),-1,R_j(x+q))$ for $q=0,1,2$, and set
$x_*=m+j-4$.  Then $\mathsf N_j(x_*)$ and $\mathsf N_j(x_*+1)$ are the submatrices of $\mathsf M_j$ on the
consecutive row triples $(j-3,j-2,j-1)$ and $(j-2,j-1,j)$,
respectively.  It is enough to show that their determinants do not
vanish simultaneously.

Write $\Delta T(x)=T(x+1)-T(x)$.  Subtracting the first row of
$\mathsf N_j(x)$ from the other two gives
\begin{equation}\label{rank_difference_determinant}
 \det \mathsf N_j(x)=
 \Delta L_j(x)\Delta R_j(x+1)
 -\Delta L_j(x+1)\Delta R_j(x).
\end{equation}
To factor the two finite differences, set
\[
\begin{aligned}
 E_j(x)&=-2jn+jx-j+4n^2-4nx+6n+x^2-4x+1,\\
 T_j(x)&= 2jn-jx+3j+4n^2-4nx+10n+x^2-6x+5.
\end{aligned}
\]
Using
$\fall{x+1}{j}=(x+1)\fall{x}{j-1}$ and
$\fall{x}{j}=(x-j+1)\fall{x}{j-1}$, one obtains
\[
 \Delta L_j(x)=
 \frac{(-1)^j j\fall{x}{j-1}E_j(x)}
 {\fall{n}{j}(2n+2-x)(2n+1-x)},
\]
\[
 \Delta R_j(x)=
 \frac{(-1)^{j+1}j(n+1)\fall{x}{j-1}T_j(x)}
 {2^j\rise{n+2-x}{j}(n+1-x)(2n+2-x)(2n+1-x)}.
\]
Substituting these expressions into
\eqref{rank_difference_determinant} and taking a common denominator,
the remaining numerator reduces by the identity
\[
\begin{split}
&E_j(x)T_j(x+1)(n+j+1-x)
 -(n-x)E_j(x+1)T_j(x)\\
&\qquad=(j+1)(2n+2-x)(2n+1-x)\Pi_j(x),
\end{split}
\]
where
\begin{equation}\label{detfactor}
\begin{aligned}
 \Pi_j(x)&=4n^2-4nx+8n-j^2+j+x^2-7x,\\
 \det \mathsf N_j(x)&=\Gamma_j(x)\Pi_j(x),\\
 \Gamma_j(x)&=
 -\frac{j^2(j+1)(n+1)(x+1)\fall{x}{j-1}^{2}}
 {2^j\fall{n}{j}\rise{n+2-x}{j}
 (2n-x)(n-x)(2n-x+1)(n-x+1)(x-j+2)}.
\end{aligned}
\end{equation}

We now specialize to $x=x_*+\ell$, $\ell=0,1$.  Every factor in
$\Gamma_j(x)$ is then a unit in $\F$.  Indeed, using
$4\le j\le b$, $3\le m\le p-2b-1$, and $n=m+b-2$, we have
\[
\begin{gathered}
 1\le x-j+2=m-2+\ell,\qquad x+1<p,\\
 2\le n-j+1\le n+1<p,\\
 3\le n+2-x\le n+j+1-x\le b+3<p,\\
 1\le n-x=b-j+2-\ell\le b-2,\\
 1\le 2n-x<2n-x+1<p,
\end{gathered}
\]
and also $1\le j<j+1<p$.  These bounds cover every factor occurring
in the numerator and denominator of $\Gamma_j(x)$.

Finally,
\begin{equation}\label{detdifference}
 \Pi_j(x_*+1)-\Pi_j(x_*)=-2(m+2b-j+3)\ne0,
\end{equation}
because
$1\le m+2b-j+3\le p-j+2\le p-2$.  Hence
$\Pi_j(x_*)$ and $\Pi_j(x_*+1)$ cannot both vanish.  By
\eqref{detfactor}, at least one of the two consecutive $3\times3$
minors is nonzero.  Therefore $\operatorname{rank}\mathsf M_j\ge3$; since
$\mathsf M_j$ has three columns, $\operatorname{rank}\mathsf M_j=3$.
\end{proof}

\begin{corollary}\label{uniquecontinuation}
Assume \eqref{rangealllarge} and $4\le j\le b$.  Once the preceding
states $\mathbf x_1,\ldots,\mathbf x_{j-1}$ are fixed, the system
\[
 \mathsf M_j\mathbf x_j=\mathbf v_j
\]
has at most one solution.  If $\mathbf v_j=0$, its unique solution is
$\mathbf x_j=0$.
\end{corollary}
\begin{proof}
The explicit formula in \eqref{coefficient_system} fixes $\mathbf v_j$
from the preceding states.  By Lemma~\ref{rank},
$\ker\mathsf M_j=\{0\}$, which proves both assertions.
\end{proof}

\begin{corollary}\label{supportpropagation}
Assume \eqref{rangealllarge}, let $4\le j\le b$, and suppose that, for
some $d\ge2$,
\[
 \mathbf x_s=0\qquad(1\le s<j,\ d\nmid s).
\]
If $d\nmid j$, then $\mathbf v_j=0$.  Consequently, every solution
$\mathbf x_j$ of \eqref{coefficient_system} is zero.
\end{corollary}
\begin{proof}
The matrix definitions make every entry of $\mathbf v_j$
weighted-homogeneous of weight $j$ in the preceding scalar coefficients,
as verified in the proof of Proposition~\ref{largegaprecursion}; this
property does not require the history to extend to a full solution.
Under the hypothesis, every surviving monomial has weight divisible by
$d$.  Since $d\nmid j$, every such weight-$j$ monomial vanishes, so
$\mathbf v_j=0$.  Corollary~\ref{uniquecontinuation} now gives
$\mathbf x_j=0$ for every solution of the current system.
\end{proof}

The first three levels therefore reduce the possible initial data to
three disjoint branches: $t\ne0$; $t=0$, $\sigma\ne0$; and
$t=\sigma=0$.  In the first branch, the third-level compatibility is
parametrized by $\sigma=\Delta_3\xi$ and $u_3=-\vartheta_3t\xi$.
The other two branches distinguish $\sigma\ne0$ from $\sigma=0$ when
$t=0$.
Lemma~\ref{rank} and Corollaries~\ref{uniquecontinuation}
and~\ref{supportpropagation} control continuation for $4\le j\le b$.  It remains to determine which of these three low-order branches
extend to full split solutions.

\subsection{Finite compatibility and branch elimination}\label{eliminationdetails}

We now test which of the three low-order branches identified above
can extend to a full solution of \eqref{witnessidentity}, with the
centering and degree conventions of Subsection~\ref{recurrence} and
with all three root polynomials split over $\F$ with distinct roots.
The tests below are necessary for such an extension; compatibility at
one level alone is not sufficient for a full solution.

For $1\le j\le b$, Proposition~\ref{largegaprecursion} gives
\[
 \mathsf M_j\mathbf x_j=\mathbf v_j,
 \qquad
 \mathbf v_j=\mathsf Q_j\mathbf h_j+(-1)^{j+1}\mathsf R_j\mathbf c_j.
\]
For $4\le j\le b$, Lemma~\ref{rank} gives full column rank, so the
current linear system has a solution exactly when
$\mathbf v_j\in\operatorname{im}\mathsf M_j$, and that solution is
unique.  It must also satisfy the prescribed degree conditions.
For $j>b$ we do not use these $j$-row matrices: we apply
Lemmas~\ref{companionremainder} and~\ref{toeplitzdivision} to the actual
degree-$b$ polynomial, retaining only the rows $0\le i<b$ and setting
$u_j=0$.  The actual-degree equations and their row normalization are
given in Appendix~\ref{app:degreeboundary}; the two endpoint calculations
are in Appendix sections~\ref{app:cubic} and~\ref{app:quartic}.

Set $\Psi_4=z^2+z+2(b-4)$.  For $b\ge4$, evaluating
$\mathbf v_4=\mathsf Q_4\mathbf h_4-\mathsf R_4\mathbf c_4$ gives the
first compatibility test:
\begin{align}
 \xi\bigl(\Delta_3^2\Psi_4\xi+4\Theta_bz(z+2)t^2\bigr)&=0
                       &&(t\ne0),\label{mregfourthnonzero}\\
 \Psi_4\sigma^2&=0&&(t=0).\label{mregfourthzero}
\end{align}
For $b\ge4$ these are exactly the compatibility conditions of the full
fourth-level system.  For $b=3$, both conditions remain necessary after
imposing $u_4=0$ in the actual-degree equations.  The specialization is
justified in Appendix~\ref{app:degreeboundary}, and the factorizations
and unit checks are given in Appendix~\ref{app:fourth}.

For $t\ne0$, \eqref{mregfourthnonzero} gives either $\xi=0$ or
\begin{equation}\label{mregxicandidate}
 \Delta_3\Psi_4\ne0,
 \qquad
 \xi=-\frac{4\Theta_bz(z+2)t^2}{\Delta_3^2\Psi_4}.
\end{equation}
For $t=0$, the second condition says that the branch $\sigma\ne0$
can survive level four only on $\Psi_4=0$.  We now show that the only
branch that survives all finite compatibility tests is the reference
branch, apart from the quartic characteristic-$13$ exception.

\begin{lemma}\label{branchscreening}
Assume \eqref{rangealllarge} and \eqref{witnessidentity}, with
$u_1=0$, and suppose that $F_B,F_{A^c},F_H$ split over $\F$ with
distinct roots.  The three low-order branches satisfy the following conclusions.
\begin{enumerate}[label=\textup{(\roman*)},nosep]
\item If $t\ne0$, then $\xi=0$.
\item There is no solution with $t=0$ and $\sigma\ne0$.
\item If $t=\sigma=0$, then necessarily $p=13$, $b=m=4$, and for some
$v\ne0$,
\begin{equation}\label{quartictriple}
 F_B=X^4+vX,\qquad F_H=Y^4+vY,\qquad
 F_{A^c}=Z^6+7vZ^3+5v^2.
\end{equation}
\end{enumerate}
\end{lemma}
\begin{proof}
Suppose first that $t\ne0$ and $\xi\ne0$.  Then
\eqref{mregxicandidate} holds, and by a common dilation we may normalize
$t=1$.  If $b=3$, the actual-degree coefficient equations at levels
four and five give the residual \eqref{cubicfifth} in
Appendix~\ref{app:cubic}, which must vanish:
\[
 -\frac{48m(m-2)(m-1)(m+3)(m+5)(m+6)}
 {7(m+1)^4(m+2)^3(m+4)^4}=0.
\]
Here $3\le m\le p-7$, so every displayed factor is a unit, a
contradiction.

Let $b\ge4$.  The fourth-level system uniquely determines
$\mathbf x_4$.  For $b\ge5$, insert this state into $\mathbf v_5$ and
test the compatibility of $\mathsf M_5\mathbf x_5=\mathbf v_5$.
For $b=4$, use instead the actual-degree fifth-level equations with
$u_5=0$.  The calculation in Appendix~\ref{app:fifth} gives in both cases the two
necessary conditions $F_5=G_5=0$, where
\[
\begin{aligned}
 \Delta_0&=z^2+2z-18,\\
 F_5&=(z+6)\Delta_0+(m-3)(z-10),\\
 G_5&=(z+14-m)\Delta_0+(m-3)(4z-37).
\end{aligned}
\]
The normalization and factorization of the residuals are given in
\eqref{U5}--\eqref{vfifthres}.  Now
$(2z^2-4z+38)F_5-(20-2z)G_5=2(z+1)(z+2)^2\Delta_0$.
Thus $\Delta_0=0$, and then $F_5=G_5=0$ forces $m=3$.  But
$\Delta_3=\Delta_0-3(m-3)$, contradicting $\Delta_3\ne0$ in \eqref{mregxicandidate}.  Hence
$\xi=0$, proving \textup{(i)}.

Now suppose $t=0$ and $\sigma\ne0$.  Equation~\eqref{mregfourthzero}
gives $\Psi_4=0$.  For $b=3,4$ one has respectively
$\Psi_4=(z-1)(z+2)$ and $\Psi_4=z(z+1)$, both nonzero in the admissible
range.  Hence $b\ge5$ and $p\ge17$.  Combining $\Psi_4=0$ with
$z=m+2b-3$ gives, in $\F$, $b=4-z(z+1)/2$ and
$m=z^2+2z-5$, and therefore
\[
 \Delta_2=4-3z,
 \qquad
 \Delta_3=-2(z-1)(z+3)\ne0.
\]
The third-level system then gives $\mathbf x_3=0$.  The fourth-level
system determines a state of the form
$\mathbf x_4=\sigma^2\mathbf d_4$, with the vector $\mathbf d_4$ defined in
\eqref{even_fourth_state}--\eqref{even_fourth_coefficients} of
Appendix~\ref{app:even}.  Since all nonzero preceding indices are even,
Corollary~\ref{supportpropagation} gives $\mathbf x_5=0$.
For $b\ge6$, evaluate $\mathbf v_6$ and test the sixth-level matrix
system; for $b=5$, use the actual-degree equations with $u_6=0$.
The residual calculation in Appendix~\ref{app:even} gives
$K_6=z^2-5z-26=0$ and $S_6=57z+182=0$.  These are the conditions \eqref{vsixthres}, obtained from the normalized
residuals \eqref{even_residual_polynomials}.  But
$57^2K_6-(57z-467)S_6=520$, which is nonzero in $\F$ because $p\ge17$.  This proves \textup{(ii)}.

Finally assume $t=\sigma=0$.  If $u_3=0$, then the third-level system
and the nonzero $c_3,h_3$ minor give $\mathbf x_3=0$.  Hence all
coefficient states through degree $b$ vanish by
Corollary~\ref{uniquecontinuation}, and $F_B=X^b$, contradicting the
distinctness of its roots.  Thus put $v=u_3\ne0$.  The third-level
compatibility \eqref{mregthird} gives $\Delta_3=0$.

For $b=3$, this is impossible because $\Delta_3=z(z-1)\ne0$.  If
$b\ge5$, Corollary~\ref{supportpropagation} gives
$\mathbf x_4=\mathbf x_5=0$.  When $b=5$ this already yields
$F_B=X^5+vX^2=X^2(X^3+v)$, which has a repeated root.  For $b\ge6$, the sixth-level compatibility
calculation in Appendix~\ref{app:purethird}, specifically
\eqref{purethirdprofile}--\eqref{vK6}, reduces to
\[
 \widetilde\kappa_6v^2=0,
 \qquad
 \widetilde\kappa_6=
 \frac{27(z+6)^2m(m^2-1)}
 {(z-6)(z-4)^2(z-3)b^2(b+1)^2(b+3)}.
\]
On $\Delta_3=0$ every factor in this quotient is a unit in the
large-gap range, so this is impossible.

It remains to take $b=4$.  Here $\Delta_3=m^2+9m+26=0$ and the
fourth state is zero.  The historical right-hand sides in the actual
degree-$4$ equations at level five vanish; the two independent
$c_5,h_5$ columns give $c_5=h_5=0$.  Appendix~\ref{app:quartic}
then computes the actual-degree sixth-level system
\eqref{quartic_sixth_table}, whose residuals \eqref{quarticsix} include
$-351v^2/3200=0$.  Since $p\ge13$ and $v\ne0$, this forces $p=13$.  The large-gap bound
then gives $3\le m\le4$, and $\Delta_3=0$ excludes $m=3$, so $m=4$.
Substitution in the third- and sixth-level equations gives precisely
\eqref{quartictriple}.  This proves \textup{(iii)}.
\end{proof}

\begin{proposition}\label{coefficientrigidity}
Assume \eqref{rangealllarge}.  Let $F_B,F_{A^c},F_H$ be monic
polynomials that split over $\F$ with distinct roots, of degrees
$b,n,m$, respectively.  Suppose that they satisfy
\eqref{witnessidentity} and $u_1=0$.  Exactly one of the following
alternatives holds:
\begin{enumerate}[label=\textup{(\roman*)},nosep]
\item $t\ne0$, $\xi=0$, and
$\mathbf x_j=\mathbf x_j^0$ for $1\le j\le b$.
In particular, $F_B=F_{B^0}$ for \eqref{referenceprogressions}.
\item $t=\sigma=0$, $p=13$, $b=m=4$, and, for some $v\ne0$,
\[
 F_B=X^4+vX,\qquad F_H=Y^4+vY,\qquad
 F_{A^c}=Z^6+7vZ^3+5v^2.
\]
\end{enumerate}
\end{proposition}
\begin{proof}
Lemma~\ref{branchscreening} leaves only the two alternatives in
the statement.  If $t\ne0$, then $\xi=0$.  The formulas from the first
three levels therefore give
\[
 \mathbf x_s=\mathbf x_s^0\qquad(1\le s\le3).
\]
Suppose, for contradiction, that $j\le b$ is the first index for which
$\mathbf x_j\ne\mathbf x_j^0$.  Then $j\ge4$, and the preceding states
of the two solutions agree.  Hence Proposition~\ref{largegaprecursion}
gives the same vector $\mathbf v_j$ for both, while
Corollary~\ref{uniquecontinuation} says that
$\mathsf M_j\mathbf x_j=\mathbf v_j$ has at most one solution.  This is
a contradiction.  Thus
$\mathbf x_j=\mathbf x_j^0$ for all $1\le j\le b$, proving
\textup{(i)}.

If $t=0$, Lemma~\ref{branchscreening}(ii) excludes
$\sigma\ne0$, while part~\textup{(iii)} gives exactly the quartic
characteristic-$13$ alternative in \textup{(ii)}.  The two alternatives
have different values of $t$, so they are disjoint.
\end{proof}

This completes the coefficient part of the classification.
Subsection~\ref{largegapinterior} next recovers the set pair from the
identified polynomial $F_B$: it uses the arithmetic-progression inverse
theorem in the reference case and the original reflection constraints
in the exceptional case.  Appendix~\ref{app:finitechecks} contains the explicit finite
compatibility calculations used above.

\section{Classification of critical pairs}\label{classification}

We now prove Theorem~\ref{main}.  Let $(A,B)$ be a critical pair with
$a\ge b$, and retain $m=|\F\setminus(A\rs B)|$.  For critical pairs,
$m\ge3$, $m=2$, $m=1$, and $m=0$ correspond respectively to
$a+b\le p-1$, $a+b=p$, $a+b=p+1$, and $a+b\ge p+2$.
We first settle the large-gap interior by the coefficient theory of
Section~\ref{structure}.  We then treat the remaining unequal-cardinality
interior in two stages: first the small summands $b=1,2$, and then the
small gaps $a-b\in\{1,2\}$ for $b\ge3$.  The latter is reduced to a single
key containment lemma.  The final subsection collects the equal-cardinality
case, the two boundary layers, and saturation, citing existing inverse
theorems whenever they cover the required range.

\subsection{The large-gap interior case}\label{largegapinterior}

Assume $m\ge3$, $b\ge3$ and $a\ge b+3$.
Proposition~\ref{coefficientrigidity} has already identified $F_B$:
it is either the root polynomial of the centered reference progression
or the exceptional quartic in characteristic $13$.  In the first case,
the following known inverse result recovers $A$; in the second, the
reflection centers determine $A^c$ directly.

\begin{lemma}\label{allAPbridge}
Suppose $b\ge3$, $a\ge b+3$, $a+b\le p-1$, and
$|A\rs B|=a+b-2$.  If $B$ is an arithmetic progression, then $(A,B)$
is a common affine image of $([0,a-1],[0,b-1])$.
\end{lemma}
\begin{proof}
Apply \cite[Theorem 1.5(ii)]{LQ} with $k=a$ and $l=b$.
Its hypotheses are $b\ge3$, $a\ge b+1$, and
$|A\rs B|=a+b-2\le p-2$, together with the progression assumption
on $B$; all follow from our stronger hypotheses.  Its conclusion gives
a common difference, with $B$ occupying the first or last $b$ positions
of $A$.
Translation and dilation give the first normal form, and a simultaneous
reversal changes the terminal-segment form into the same form.
\end{proof}

\begin{proposition}\label{lg}
Let $A,B\subseteq\F$ have cardinalities $a,b$ satisfying
\[
 b\ge3,\qquad a\ge b+3,\qquad a+b\le p-1.
\]
Then $(A,B)$ is critical if and only if, up to common affine equivalence,
it is one of
\[
 ([0,a-1],[0,b-1]),\qquad (A_*,B_*),
\]
where $(A_*,B_*)$ is the representative in \eqref{exception}.
\end{proposition}

\begin{proof}
Assume first that $(A,B)$ is critical.  Center $B$ and let $H$ be the
halves of the missing restricted sums.  The basic reflection implication
of Subsection~\ref{reflectionclosure} gives
$2h-x\in A^c$ for $x\in B$, $h\in H$, $x\ne h$.  After the common
translation used to center $B$, the same reflection condition is
preserved.  Forming the three root polynomials, Lemma~\ref{gridlemma}
therefore gives \eqref{witnessidentity}, and
Proposition~\ref{coefficientrigidity} applies.

In alternative \textup{(i)}, the proposition gives $F_B=F_{B^0}$.
Since the roots are distinct, $B=B^0$ is an arithmetic progression.
Lemma~\ref{allAPbridge} therefore gives the common-endpoint pair.
No further coefficient comparison is required: the matrix recursion
and branch elimination were completed in
Proposition~\ref{coefficientrigidity}.

\emph{Recovery of the exceptional branch.}
In the second alternative, $p=13$ and $b=m=4$, hence $a=7$.  Choose a
nonzero root $r$ of $X^4+vX$.  Then $r^3=-v$, and the common dilation
by $r^{-1}$ transforms the monic root polynomial of $B$ into $X^4-X$;
therefore $B=\{0,1,3,9\}$.  Alternative \textup{(ii)} of Proposition~\ref{coefficientrigidity}
also gives $F_H(Y)=Y^4+vY$ before dilation.  Hence the same normalization
gives $H=B$.
The reflection constraints now force
\[
 \{2h-x:h,x\in B,\ x\ne h\}
       =\{2,4,5,6,10,12\}\subseteq A^c.
\]
Since $|A^c|=6$, equality holds, and hence
$A=\{0,1,3,7,8,9,11\}$.  The map $x\mapsto2x+1$ sends the representative in \eqref{exception}
to this pair.  Thus the exceptional polynomial determines exactly the orbit
$\mathcal E$; the reflection constraints recover $A^c$ directly.

For sufficiency, let $A=[0,a-1]$ and $B=[0,b-1]$.
For $1\le s\le a-1$, use the representation $s=s+0$.
For $a\le s\le a+b-2$, use
\[
 s=(a-1)+(s-a+1),\qquad 1\le s-a+1\le b-1<a-1.
\]
Both representations have unequal summands.  The only representation of $0$ is excluded,
and no sums outside the displayed interval occur.  As $a+b\le p-1$, this gives
$A\rs B=[1,a+b-2]$ without wraparound.  For the displayed exceptional
pair, direct addition gives
\[
 A\rs B=\{1,3,4,7,8,9,10,11,12\},
\]
of size $9=7+4-2$.  Finally, common affine transformations preserve
criticality, proving both directions.
\end{proof}

The progression alternative and the exceptional orbit have thus been
established without classifying all auxiliary polynomials.  Outside
these two alternatives, \eqref{anr} is strict and gives
$|A\rs B|\ge a+b-1$.  The exceptional alternative is absent whenever
$p\ge17$ or $a+b\le p-3$.

\subsection{Small summands and small gaps}\label{smallgapinterior}

The unequal-cardinality interior not covered by Proposition~\ref{lg}
consists of the small summands $b=1,2$ and, for $b\ge3$, the two gaps
$a-b=1,2$.  We handle the small summands first because the case $b=2$
arises when a small-gap pair with $B\nsubseteq A$ is compressed.

\begin{proposition}\label{smallcard}
Suppose $a+b\le p-1$, $a>b$, and $b\in\{1,2\}$.  If $b=1$, then
$(A,B)$ is critical if and only if, up to common affine equivalence,
$B=\{0\}$ and $0\in A$.  If $b=2$, then $(A,B)$ is critical if and
only if, up to common affine equivalence, $B=\{0,1\}$ and
\[
 A=[0,a-1]\quad\text{or}\quad A=[2-a,1],
\]
or
\[
 A=\{0,1\}\cup[\ell,\ell+a-3],
 \qquad 3\le\ell,\quad \ell+a-3\le p-2.
\]
\end{proposition}
\begin{proof}
If $B=\{v\}$, then $A\rs B$ is obtained from the translate $A+v$ by
deleting $2v$ exactly when $v\in A$.  Hence
$|A\rs B|=a-1$ if and only if $v\in A$, which is precisely $B\subsetneq A$.  Translating by $-v$ gives the
stated affine normal form $B=\{0\}$ with $0\in A$.

Now let $B=\{0,1\}$, and let $\varrho$ be the number of maximal
consecutive blocks of the proper set $A$ in the cyclic order with
step one.  We call these blocks runs.  Then
$|A\cup(A+1)|=a+\varrho$.  Restriction can remove only $0$ and $2$,
with respective indicators
\[
 \varepsilon_0=\ind_{\{0\in A,\;-1\notin A\}},\qquad
 \varepsilon_2=\ind_{\{1\in A,\;2\notin A\}}.
\]
Thus criticality is equivalent to $\varrho=\varepsilon_0+\varepsilon_2$.
Since $A$ is nonempty and proper, $\varrho\ge1$, while the right-hand
side is at most $2$.  If $\varrho=1$, exactly one indicator is one, so
the unique run starts at $0$ or ends at $1$, giving the first two forms.
If $\varrho=2$, both indicators are one; one run is then $\{0,1\}$ and
the other is $[\ell,\ell+a-3]$ with the displayed bounds.  The converse
follows from the same identity.
\end{proof}

We now assume $a+b\le p-1$, $b\ge3$, and $a-b\in\{1,2\}$.
Once $B\subsetneq A$ is known, the deletion patterns are already covered
for $b\ge4$ by Liu and Qian.  The only new structural point is therefore
the following containment statement.

\begin{lemma}\label{containment}
Suppose $(A,B)$ is critical, $a+b\le p-1$, $b\ge3$, and
$a-b\in\{1,2\}$.  Then $B\subsetneq A$.
\end{lemma}
\begin{proof}
Suppose instead that $B\nsubseteq A$.  Put
$ I=A\cap B$, $U=A\cup B$, $D_A=A\setminus B$, and
$D_B=B\setminus A$, and write $c=|I|$ and $N=|U|$.
Since $D_B\ne\varnothing$ and $a>b$, both $D_A$ and $D_B$ are nonempty.
Moreover
\begin{equation}\label{smallgapcardinality}
 N+c=a+b\le p-1,
 \qquad
 N-c=a-b+2|D_B|\ge3.
\end{equation}

If $c=0$, then the restriction removes no representation and
$A\rs B=A+B$.  Cauchy--Davenport gives
$|A\rs B|\ge a+b-1$, contradicting criticality.

Assume henceforth that $c\ge1$.  If $x\in U$ and $y\in I$ are distinct,
then $x+y$ is an admissible sum in $A\rs B$, with the order interchanged
when necessary.  Hence $U\rs I\subseteq A\rs B$.  Since $N>c$ by
\eqref{smallgapcardinality}, the mixed lower bound gives
$|U\rs I|\ge N+c-2=a+b-2=|A\rs B|$.  Therefore
\begin{equation}\label{smallgapcompression}
 U\rs I=A\rs B,
 \qquad
 D_A+D_B\subseteq U\rs I.
\end{equation}
In particular, if two distinct points of
$W:=U\setminus I=D_A\sqcup D_B$ have a sum outside $U\rs I$, then they must belong to the same part of
the partition $D_A\sqcup D_B$.  We shall use the missing sums of the
compressed pair to propagate this constraint through $W$.

Suppose first that $c=1$, say $I=\{v\}$.  The only ordinary sum that can
be removed by restriction is $2v$, so Cauchy--Davenport and criticality
give
\[
 |A+B|=a+b-1\le p-2,
 \qquad
 (A+B)\setminus(A\rs B)=\{2v\}.
\]
Vosper's theorem \cite{Vosper1,Vosper2} therefore makes $A$ and $B$
arithmetic progressions with a common difference.  Their only sums with
a unique ordered representation are the two endpoint sums.  Since the
unique lost sum is $2v$ with $v\in A\cap B$, the point $v$ must be a common initial or terminal point.  As $a>b$, this
forces $B\subsetneq A$, a contradiction.

Now let $c=2$.  Normalize $I=\{0,1\}$.  By
\eqref{smallgapcompression}, $(U,I)$ is critical, and
Proposition~\ref{smallcard} gives its three possible forms.  The second
interval form is carried to the first by the common reflection
$x\mapsto1-x$, so first take
\[
 U=[0,N-1].
\]
Then
\[
 U\rs I=[1,N],\qquad W=[2,N-1].
\]
Put $L=N-2\ge3$ and write the vertices of $W$ as $2+i$,
$0\le i<L$.  The two sums $N+1,N+2$ are missing, and on the index set
$[0,L-1]$ they correspond to the pair sums $L-1,L$.  The chain
\[
 0,\ L-1,\ 1,\ L-2,\ 2,\ L-3,\ldots
\]
visits every vertex.  Consecutive vertices therefore have a missing sum
and hence, by \eqref{smallgapcompression}, must lie in the same one of
$D_A,D_B$.  Thus all of $W$ lies in one part, contradicting the
nonemptiness of both parts.

It remains in the case $c=2$ to consider
\[
 U=\{0,1\}\cup[\ell,r],
 \qquad 3\le\ell\le r\le p-2.
\]
Here $W=[\ell,r]$, $L=r-\ell+1\ge3$, and
\[
 U\rs I=\{1\}\cup[\ell,r+1].
\]
Depending on whether $\ell+r\le p$, $\ell+r=p+1$, or
$\ell+r\ge p+2$, choose the following two integers, whose residues
are missing sums:
\[
\begin{array}{c|c|c}
\text{condition}&\text{missing sums}&\text{index sums on }[0,L-1]\\ \hline
\ell+r\le p&\ell+r-1,\ \ell+r&L-2,\ L-1\\
\ell+r=p+1&p,\ p+2&L-2,\ L\\
\ell+r\ge p+2&\ell+r,\ \ell+r+1&L-1,\ L.
\end{array}
\]
In the first row the chosen sums lie beyond $r+1$ and at most $p$; in
the middle row their residues are $0,2$; and in the third row their
residues lie below $\ell$.  Thus all six displayed sums are indeed
outside $U\rs I$.  For the first and third rows, the two reflections
generate a spanning chain obtained from
\[
 0,\ L-1,\ 1,\ L-2,\ldots
\]
by reversing the indices when necessary.  In the middle row
$L=p+2-2\ell$ is odd, and
\[
 0,\ L-2,\ 2,\ L-4,\ 4,\ldots,\ L-1
\]
is a spanning chain whose consecutive sums alternate between $L-2$ and
$L$.  Thus the missing-sum reflections again force all points of $W$ to
lie in the same part of $D_A\sqcup D_B$, a contradiction.

Finally suppose that $c\ge3$.  By \eqref{smallgapcardinality}, the
critical compressed pair $(U,I)$ lies in the large-gap range, so
Proposition~\ref{lg} applies.  In the progression alternative, after a
common affine normalization,
\[
 U=[0,N-1],\qquad I=[0,c-1].
\]
Then
\[
 U\rs I=[1,N+c-2],\qquad W=[c,N-1].
\]
With $L=N-c\ge3$, the two missing sums $N+c-1,N+c$ correspond after
translating $W$ by $-c$ to the index sums $L-1,L$.  The same spanning
chain as above forces all of $W$ into one of $D_A,D_B$, again impossible.

In the exceptional alternative, up to common affine equivalence
$(U,I)=(A_*,B_*)$.  Then $W=\{3,5,10\}$, and the three pair sums
\[
 3+5=8,\qquad 3+10=0,\qquad 5+10=2\pmod{13}
\]
are all missing from $U\rs I$.  Hence the three points of $W$ must all
lie in the same part of $D_A\sqcup D_B$, a final contradiction.
Thus $B\subsetneq A$.
\end{proof}

With containment established, the classification is immediate from the
known deletion theorem except at the two smallest endpoints.

\begin{proposition}\label{smallgapclassification}
Assume $a+b\le p-1$ and $b\ge3$.
\begin{enumerate}[label=\textup{(\roman*)},nosep]
\item If $a=b+1$, then $(A,B)$ is critical if and only if, up to common
affine equivalence, either
\[
 A=[0,b],\qquad B=A\setminus\{r\},\qquad 0\le r\le b,
\]
when $b\ge4$, or
\[
 A=2h-A,\qquad B=A\setminus\{v\},\qquad h\in\F,\ v\in A,
\]
when $b=3$.
\item If $a=b+2$, then $(A,B)$ is critical if and only if, up to common
affine equivalence,
\[
 A=[0,b+1],\qquad B=A\setminus T,
\]
where
\[
 T\in\bigl\{\{0,1\},\{0,2\},\{b-1,b+1\},\{b,b+1\}\bigr\}.
\]
\end{enumerate}
\end{proposition}
\begin{proof}
For necessity, Lemma~\ref{containment} gives $B\subsetneq A$.
When $b\ge4$, apply \cite[Proposition 3.4]{LQ} with $k=a$ and $l=b$.
Its hypotheses hold because $a-b\in\{1,2\}$ and
$|A\rs B|=a+b-2\le p-3$.  It follows that $A$ is an arithmetic
progression and $B$ has one of the displayed deletion forms.

For sufficiency in this range, normalize $A=[0,a-1]$.
Then $A\rs A=[1,2a-3]$, with no reduction modulo $p$ needed since
$2a-3<p$.  Deleting a single element to form $B$ leaves every
unordered pair of distinct elements represented in $A\rs B$, so
$|A\rs B|=2a-3=a+b-2$.  If $B=A\setminus T$ and $T$ is one of the
four displayed two-element sets, the only unordered pair no longer
represented is $T$ itself.  Its sum is one of
$1,2,2a-4,2a-3$, each of which has a unique unordered representation
by distinct elements of $A$.  Thus exactly one sum is lost and
$|A\rs B|=2a-4=a+b-2$.

It remains to treat $b=3$.  The four-element self-sum criterion used
in the first case is already noted in \cite[Section 1]{Karolyi}; we
recall its elementary proof.

If $(a,b)=(4,3)$, write $A\setminus B=\{v\}$.  Then
$A\rs B=A\rs A$.  Among the six unordered pairs of a four-set, two
pairs sharing a vertex cannot have the same sum, while two distinct
collisions between disjoint partitions would force two elements of
$A$ to coincide.  Hence $|A\rs A|=5$ exactly when one partition into two pairs
has equal pair sums, which is equivalent to central symmetry of
$A$.  Every three-subset then gives the required critical pair.

If $(a,b)=(5,3)$, write $A\setminus B=T=\{x,y\}$.  Then
$A\rs A=(A\rs B)\cup\{x+y\}$.  Criticality and the Erd\H{o}s--Heilbronn bound give
$|A\rs A|=7=2|A|-3$, so K\'arolyi's inverse theorem \cite[Theorem 7]{Karolyi}
makes $A$ a five-term arithmetic progression.  After normalization
$A=[0,4]$, the restricted sums with a unique unordered representation
are precisely those belonging to
\[
 \{0,1\},\quad\{0,2\},\quad\{2,4\},\quad\{3,4\}.
\]
Thus $T$ is one of these four pairs, and conversely each choice removes
exactly its unique sum.  This proves the remaining case.
\end{proof}

\subsection{The remaining cases}\label{remainingcases}

It remains to treat equal cardinalities and the three ranges at or above
the boundary.  We use existing results wherever they apply and supply
only the few cases not covered by them.

\begin{proposition}\label{equalprop}
Let $|A|=|B|=k\ge2$ and $p>2k-1$.  If $|A\rs B|\le2k-2$, then $A=B$.
Consequently, distinct sets in this range satisfy $|A\rs B|\ge2k-1$.
\end{proposition}
\begin{proof}
Suppose $A\ne B$.  By \eqref{anr}, the hypothesis forces
$|A\rs B|=2k-2$.  For $k\ge5$, this is ruled out by
\cite[Theorem 7]{DGHM}, which gives $A=B$.

It remains to consider $2\le k\le4$.  Put $I=A\cap B$,
$U=A\cup B$, $D_A=A\setminus B$, and $D_B=B\setminus A$, and write
$c=|I|$.  If $c=0$, then $A\rs B=A+B$, and
Cauchy--Davenport gives $|A\rs B|\ge2k-1$, a contradiction.
If $c=1$, say $I=\{v\}$, then
$A+B=(A\rs B)\cup\{2v\}$.  Hence $|A+B|=2k-1\le p-2$, and Vosper's
theorem makes $A$ and $B$ equal-length arithmetic progressions with a
common difference.  Since they meet in exactly one point, after a common
affine transformation they are $[0,k-1]$ and $[k-1,2k-2]$; but then
$2v=2k-2=0+(2k-2)\in A\rs B$, a contradiction.

Assume $c\ge2$.  Then $U\rs I\subseteq A\rs B$, while
$|U\rs I|\ge |U|+|I|-2=2k-2=|A\rs B|$.  Thus
\begin{equation}\label{equalcompression}
 U\rs I=A\rs B,
 \qquad D_A+D_B\subseteq U\rs I.
\end{equation}
If $c=2$, then $k=3$ or $4$, and Proposition~\ref{smallcard}
classifies $(U,I)$.  When $k=4$, the difference set $U\setminus I$
has four elements, so the $c=2$ reflection-propagation argument in
Lemma~\ref{containment}, together with \eqref{equalcompression},
forces it into one of $D_A,D_B$, a contradiction.

For $k=3$, the difference set has only two elements, which we check
directly.  Normalize $I=\{0,1\}$.  Up to the reflection $x\mapsto1-x$,
Proposition~\ref{smallcard} gives either $U=[0,3]$ or
$U=\{0,1,\ell,\ell+1\}$ with $3\le\ell\le p-3$.
In the first case, $U\setminus I=\{2,3\}$ and
$U\rs I=[1,4]$, so its two points have the missing sum $5$.
In the second case, $U\rs I=\{1\}\cup[\ell,\ell+2]$.
If $2\ell+1<p$, then $2\ell+1>\ell+2$; otherwise its standard
residue is $2\ell+1-p<\ell$ and cannot equal $1$, since that would
imply $2\ell=0$ in $\F$.  Thus the sum of the two points of
$U\setminus I$ is again missing.  But $D_A,D_B$ are singletons in
this case, so that sum belongs to $D_A+D_B$, contradicting
\eqref{equalcompression}.

Hence $c\ge3$, which is possible here only for $k=4$ and $c=3$.
Then $(U,I)$ has sizes $(5,3)$ and is critical, so
Proposition~\ref{smallgapclassification}(ii) applies.  Its two deleted
points are exactly the two points of $U\setminus I=D_A\sqcup D_B$, and
for each of the four models their sum is missing from $U\rs I$.  This
contradicts $D_A+D_B\subseteq U\rs I$ in \eqref{equalcompression}.
Thus $A=B$.
\end{proof}

For $k=1$, distinct singletons are not critical, so the equal-cardinality
row of Table~\ref{total} is complete.

\begin{proposition}\label{twomissingclassification}
If $a+b=p$ and $a\ge b$, then $(A,B)$ is critical if and only if,
up to common affine equivalence, it is of the form \eqref{twoform}.
\end{proposition}
\begin{proof}
Since $p$ is odd, $a+b=p$ and $a\ge b$ imply $a>b$.  The result is
exactly \cite[Theorem 1.1]{HLY}.  Their normalization sends the two
missing restricted sums to $0,1$; after the common dilation
$x\mapsto2x$ these become $0,2$.  With their parameter $v$ written as
$r=v$ and $s=b-v$, their models are precisely the pairs
$\mathcal T_{r,s}$ in \eqref{twoform}.
\end{proof}

We next consider the case of a single missing restricted sum.

\begin{proposition}\label{pplus}
Assume $A\ne B$ and $a+b=p+1$.  Then $(A,B)$ is critical if and only if,
up to common affine equivalence, it is of the form \eqref{plusform} for
some $h\in B$.  For a representative satisfying \eqref{plusform}, $2h$
is the unique missing restricted sum.
\end{proposition}
\begin{proof}
For necessity, criticality gives $|A\rs B|=p-1$, so there is a unique
missing sum.  Write it uniquely as $2h$, since $p$ is odd.
The basic reflection implication from Subsection~\ref{reflectionclosure}
gives $A\cap(2h-B)\subseteq\{h\}$.  On the other hand,
$|A\cap(2h-B)|\ge a+b-p=1$.  Hence $A\cap(2h-B)=\{h\}$.  Since
$|A|+|2h-B|-1=a+b-1=p$, one also has $A\cup(2h-B)=\F$.  In particular $h\in A\cap B$ and
\[
 A^c=(2h-B)\setminus\{h\},\qquad
 A=(\F\setminus(2h-B))\cup\{h\},
\]
which is \eqref{plusform}.

Conversely, assume \eqref{plusform} for some $h\in B$.  Then
$A\cap(2h-B)=\{h\}$, so the only possible representation of $2h$ is
the excluded diagonal pair $(h,h)$.  Hence $2h\notin A\rs B$.
The mixed lower bound gives $|A\rs B|\ge p-1$, while the missing sum
gives the reverse inequality.  Thus $|A\rs B|=p-1$, and the pair is
critical with unique missing restricted sum $2h$.
\end{proof}

It remains to consider the saturated regime $a+b\ge p+2$.
Fix $s\in\F$.  The elements of $A\cap(s-B)$ are in bijection with the
ordinary representations $s=x+y$ with $x\in A$ and $y\in B$, via
$y=s-x$.  Therefore $|A\cap(s-B)|\ge |A|+|s-B|-p=a+b-p\ge2$.
Among these representations at most one is diagonal, since $x=y$ would
force the unique value $x=y=s/2$.  Thus at least one representation has
$x\ne y$, and hence $s\in A\rs B$.  Since $s$ was arbitrary,
$A\rs B=\F$.  Consequently every mixed pair with $a+b\ge p+2$ is critical.

\begin{proof}[Completion of the proof of Theorem~\ref{main}]
The large-gap interior is Proposition~\ref{lg}.  Proposition~\ref{smallcard}
handles $b=1,2$, and Proposition~\ref{smallgapclassification} handles
the two small gaps for $b\ge3$.  Proposition~\ref{equalprop} excludes
distinct equal-cardinality pairs, with $k=1$ treated above.  At the
boundary, Proposition~\ref{twomissingclassification} handles $a+b=p$,
Proposition~\ref{pplus} handles $a+b=p+1$, and the preceding counting
argument gives saturation for $a+b\ge p+2$.  These ranges are pairwise
disjoint and exhaustive, and each displayed family has been verified in
the corresponding result.  This completes the proof.
\end{proof}

\appendix
\section{Finite coefficient checks}\label{app:finitechecks}

This appendix supplies the finite identities used in
Subsections~\ref{recurrence} and~\ref{eliminationdetails}.  For
$1\le j\le b$, the right-hand side is
\[
 \mathbf v_j=\mathsf Q_j\mathbf h_j+(-1)^{j+1}\mathsf R_j\mathbf c_j
\]
from Proposition~\ref{largegaprecursion}.  At levels $j>b$ we instead
use the actual-degree equations derived from
Lemmas~\ref{companionremainder} and~\ref{toeplitzdivision} in
Appendix~\ref{app:degreeboundary}.  These equations use only the $b$ actual coefficient rows; the matrices of Proposition~\ref{largegaprecursion}
are not being extended beyond their stated range.

We fix the row notation used in the stable-level calculations.  For an
ordered tuple $I$ of distinct row indices, $\mathsf M_{j,I}$ and
$\mathbf v_{j,I}$ denote the corresponding row-submatrix of
$\mathsf M_j$ and subvector of $\mathbf v_j$, in that order.
For a single row, we use $(\mathsf M_j)_{r,*}$ and retain
$V_{r,j}=(\mathbf v_j)_r$.  Thus the level index $j$ always belongs
to the underlying matrix, and $I$ or $r$ selects its rows.

For $4\le j\le b$, the following construction avoids a case split
according to which consecutive minor is nonzero.  Put
\[
 I_0=(j-3,j-2,j-1),\qquad I_1=(j-2,j-1,j),
 \qquad \mathsf N_\ell=\mathsf M_{j,I_\ell},
\]
\[
 \mathbf b_\ell=\mathbf v_{j,I_\ell},\qquad
 \gamma_\ell=\Gamma_j(m+j-4+\ell),\qquad
 d=-2(m+2b-j+3).
\]
By \eqref{detfactor}--\eqref{detdifference}, every solution of the
current system satisfies
\begin{equation}\label{stable_candidate_formula}
 \mathbf x_j=
 \frac{\operatorname{adj}(\mathsf N_1)\mathbf b_1/\gamma_1
       -\operatorname{adj}(\mathsf N_0)\mathbf b_0/\gamma_0}{d}.
\end{equation}
Indeed, for such a solution the two terms in the numerator are
$\Pi_j(m+j-3)\mathbf x_j$ and $\Pi_j(m+j-4)\mathbf x_j$.
Their difference is $d\mathbf x_j$, and $\gamma_0,\gamma_1,d$ are
units by Lemma~\ref{rank}.  The right-hand side therefore specifies
the only possible candidate; it is a solution only if all rows of
$\mathsf M_j\mathbf x_j=\mathbf v_j$ hold.  We use this construction
in Appendix sections~\ref{app:fifth} and~\ref{app:even}.

\subsection{Low-order data used in the finite checks}\label{app:loworder}

For the reference progressions \eqref{referenceprogressions}, put
$\mu=zt/(n(z+1))$.  The coefficients through level three are
\begin{equation}\label{reference_low_explicit}
\begin{aligned}
 c_1^0&=t,\qquad h_1^0=-m\mu,\\
 c_2^0&=\frac{n-1}{2n}t^2
              -\frac{n^2-1}{6n(z+1)^2}t^2,\\
 h_2^0&=\frac{m(m-1)}2\mu^2-\frac{m(m^2-1)}{96}\delta^2,\\
 c_3^0&=\frac{(n-1)(n-2)}{6n^2}t^3
              -\frac{(n-2)(n^2-1)}{6n^2(z+1)^2}t^3,\\
 h_3^0&=-\frac{m(m-1)(m-2)}6\mu^3
           +\frac{m(m-2)(m^2-1)}{96}\mu\delta^2.
\end{aligned}
\end{equation}
They follow by multiplying the arithmetic-progression root factors; the
relevant power sums have denominators dividing $24$, hence are valid in
the present characteristic range.

We record the third-level identities used later.  Put
$k_1=-mz/(n(z+1))$, so $h_1=k_1t$.  Substitution of
\eqref{mregsecond} into \eqref{third_explicit_rows_main} gives
$\mathbf v_3-\mathbf v_3^0=t\sigma\mathbf d_3$, where
\[
\mathbf d_3=
\begin{pmatrix}
 -\gamma_2\lambda_{b-1}(3,2)\\
 -\gamma_2\lambda_{b-2}(3,2)
       +\beta_2(-\lambda_{b-2}(1,0)-k_1)\\
 \beta_2(-\lambda_{b-3}(1,0)-k_1)
       +k_1(\alpha_2\lambda_{b-3}(2,0)-\beta_2)
\end{pmatrix}.
\]
Multiplication by the row vector $(\eta_1,\eta_2,\eta_3)$ from the
third-level calculation gives $K_3\vartheta_3$ after substituting
\eqref{alpha2beta2} and \eqref{lambda}; this is the finite identity used
in deriving \eqref{mregthird}.  No division by $\Delta_2$ occurs.

On the branch $t\ne0$, the third-level deviations under
\eqref{mregxi} are
\begin{equation}\label{c3h3regular}
 c_3=c_3^0+\alpha_3t\xi,\qquad
 h_3=h_3^0+\beta_3t\xi,
\end{equation}
where
\begin{equation}\label{alpha3beta3}
\begin{aligned}
 \alpha_3&=\frac{(n-1)(n-2)(n+1)(\Delta_2+5z+2)
                   ((z+1)\Delta_3+z(z+5))}
                   {4b(b^2-1)n^2(z+1)^3},\\
 \beta_3&=-\frac{z(m-2)(m-1)m(m+1)
             (12b^2+8bz^2-68b+z^4-21z^2+92)}
             {4b(b^2-1)n^3(z+1)^3}.
\end{aligned}
\end{equation}
We verify these formulas in the full system, without selecting a fixed
two-row subsystem.  Substituting \eqref{lambda} and
\eqref{alpha2beta2} into the three components gives
\begin{equation*}
 \mathsf M_3
 \begin{pmatrix}\alpha_3\\\beta_3\\-\vartheta_3\end{pmatrix}
 =\Delta_3\mathbf d_3.
\end{equation*}
Since $\sigma=\Delta_3\xi$ and $u_3^0=0$, the vector
$\mathbf x_3^0+t\xi(\alpha_3,\beta_3,-\vartheta_3)^{\mathsf T}$
therefore satisfies all three equations.  If $\Delta_3\ne0$, the
invertibility of $\mathsf M_3$ gives uniqueness.  If $\Delta_3=0$,
\eqref{mregxi} fixes $u_3$, and the nonzero $c_3,h_3$ minor established
on that locus in Subsection~\ref{recurrence} gives uniqueness of
$c_3,h_3$.  This verification divides by neither $\Delta_2$ nor $\Delta_3$ and
does not assume that a fixed $c_3,h_3$ subminor is nonzero away from
$\Delta_3=0$.  At $\xi=0$ all three initial vectors agree with the reference.

\subsection{Actual-degree equations and one-step specialization}\label{app:degreeboundary}

We first specify the coefficient equations used when $j>b$.
Keep the actual degree-$b$ polynomial $F_B$, and let $\mathsf I_b$ be
the $b\times b$ identity matrix.  Applying
Lemma~\ref{companionremainder} coefficientwise in $Y$ gives
\begin{equation}\label{actual_degree_remainder}
 \begin{aligned}
 G_i(Y)&=-2^{-n}\mathbf e_i^{\mathsf T}
  (\mathsf C_{F_B}-Y\mathsf I_b)
  F_{A^c}(2Y\mathsf I_b-\mathsf C_{F_B})\mathbf e_0,\\
 G_i(Y)&=\sum_{s=0}^{m+r-1}g_{i,s}Y^{m+r-1-s},
 \qquad r=b-i,\quad 0\le i<b.
 \end{aligned}
\end{equation}
The leading coefficient is the nonzero $g_{i,0}$ in \eqref{leading_gi}.
Apply Lemma~\ref{toeplitzdivision} to $F_H$ and $G_i/g_{i,0}$, with
$d=r-1$.  Its top block is $\mathsf H_r$, and its quotient vector is
$\widehat{\mathbf q}_r$ from \eqref{history_quotient_vector}.
Thus its tail equations are precisely
\begin{equation}\label{actual_degree_residual}
 \mathcal D_{i,j}:=g_{i,j}
   -g_{i,0}\sum_{s=0}^{r-1}\widehat q_{r,s}h_{j-s}=0,
 \qquad r\le j\le m+r-1,
\end{equation}
where $\widehat q_{r,0}=1$.  With $g_{i,j}=0$ beyond the displayed
degree, later equations are identically zero by
\eqref{actual_degree_conditions}.  These equations use only the
actual rows $0\le i<b$ and remain valid at every needed endpoint.

To distinguish raw and normalized right-hand sides, set
\begin{equation}\label{actual_degree_rhs}
 \mathcal W_{i,j}:=
 -\left.\mathcal D_{i,j}\right|_{c_j=h_j=u_j=0}.
\end{equation}
For $r\le j\le b$, the calculation in
Proposition~\ref{largegaprecursion} gives
\begin{equation}\label{actual_degree_normalization}
 \begin{aligned}
 \mathcal D_{b-r,j}
   &=g_{b-r,0}\bigl((\mathsf M_j)_{r,*}\mathbf x_j-V_{r,j}\bigr),\\
 \mathcal W_{b-r,j}&=g_{b-r,0}V_{r,j}.
 \end{aligned}
\end{equation}
In particular, raw rows and right-hand sides are multiplied by
$g_{i,0}$, not merely reordered.  For $j>b$, $u_j=0$ is imposed from
the outset, and \eqref{actual_degree_residual} is linear in the
remaining current coefficients $c_j,h_j$, with raw right-hand side
$\mathcal W_{i,j}$.  A coefficient past its actual degree is fixed to
zero.  Formally retaining a column for such a coefficient can only
enlarge a candidate system, so inconsistency of the enlarged system also rules out an actual
solution.  We do not write these endpoint
right-hand sides as components of the undefined $j$-row vector
$\mathbf v_j$.

For the one-step endpoints $(b,j)=(3,4),(4,5),(5,6)$, the same
necessary identities can also be obtained by regular specialization.
At $b=j-1$, the range \eqref{rangealllarge} reads
$3\le m\le p-2j+1$.
Fix $j\ge4$, put $n=m+b-2$ and $\varepsilon=b-j+1$, and regard the
already evaluated coefficient entries as rational expressions in
$b,m$.  This is entrywise continuation from degrees $b\ge j$, not
an evaluation of $\mathbf e_{b-j}$ or of a variable-size companion matrix at
$b=j-1$.  In the last formal row, the only singular entry is the
$u_j$-coefficient.  Multiplying that row by $\varepsilon$ and using
\eqref{lambda} gives, at $b=j-1$,
\[
 \tau_j u_j=0,
 \qquad
 \tau_j=\frac{(-1)^{j+1}\fall{n+2}{j}(n+j)}
                  {2^j(j-1)!\,n}.
\]
The historical reduction entries on this row are zero: fewer than $j$
steps cannot return to $X^{b-j}$, and the historical polynomial omits
$u_j$.  Its quotient entries involve only $\lambda_{b-j}(k,0)$ with
$k<j$, which are regular at $\varepsilon=0$.  Hence the historical right-hand
side is killed by the same factor, and the scaled last row becomes
exactly $u_{b+1}=0$.

After $u_j=0$, the remaining rows specialize to
\eqref{actual_degree_residual}, up to the nonzero row scalings in
\eqref{actual_degree_normalization}.  We verify that the regularized coefficient matrix still has rank
three.  In $\Gamma_j$ we retain
$n=m+b-2$ until cancellation; after specialization, the expressions
on the right below use $n=m+j-3$.  The first consecutive triple does
not contain the scaled last row, while the second does.  Their
prefactors are respectively
\begin{equation}\label{boundary_prefactors}
\begin{aligned}
 \gamma_0^*
 &=\left.\Gamma_j(m+j-4)\right|_{b=j-1}
   =-\frac{j^2(j+1)\fall{n-1}{j-2}}
            {2^j(j+2)!(n+2)},\\
 \gamma_1^*
 &=\left.(b-j+1)\Gamma_j(m+j-3)\right|_{b=j-1}
   =-\frac{j^2(j+1)(n+1)\fall{n-1}{j-3}}
            {2^j(j+1)!(m-2)}.
\end{aligned}
\end{equation}
In the second evaluation the factor $b-j+1=n-x$ in
\eqref{detfactor} is cancelled before specialization.  The displayed
simplifications use
$\rise{3}{j}=(j+2)!/2$, $\rise{2}{j}=(j+1)!$, and
$\fall{n}{j}=n\fall{n-1}{j-1}$.
The resulting determinants are
$\gamma_0^*\Pi_j(n-1)$ and $\gamma_1^*\Pi_j(n)$.
The falling products in these prefactors have factors from $m-1$ to
$n-1$ and from $m$ to $n-1$, respectively.  Both prefactors are units:
from $3\le m\le p-2j+1$ we have
\[
 1\le m-2<p,\qquad
 m\le n-1=m+j-4<p,\qquad
 n+2=m+j-1\le p-j<p,\qquad j+2<p.
\]
These inequalities cover every factor in \eqref{boundary_prefactors}.
Moreover, $\Pi_j(n)-\Pi_j(n-1)=-2(m+j+1)\ne0$, because $0<m+j+1\le p-j+2<p$.  Thus at least one regularized
minor is nonzero.  Finally, the factors of $\fall{n+2}{j}$ run from
$m$ to $m+j-1$, and $0<n+j=m+2j-3\le p-2$, so $\tau_j$ is a unit.
This proves the rank and unit assertions needed for the one-step
specialization of the necessary identities.  More distant
endpoint levels are computed directly from
\eqref{actual_degree_remainder}--\eqref{actual_degree_residual}, as in
Appendix sections~\ref{app:cubic} and~\ref{app:quartic}; no stable-rank claim
is made for $j>b$.

\subsection{Fourth-level compatibility}\label{app:fourth}

Abbreviate $\lambda_r^{k,\ell}=\lambda_{b-r}(k,\ell)$.  Direct
evaluation of
\[
 \mathbf v_4=\mathsf Q_4\mathbf h_4-\mathsf R_4\mathbf c_4
\]
from Proposition~\ref{largegaprecursion} gives, for $b\ge4$,
\begin{equation}\label{fourth_rhs}
\begin{aligned}
 V_{1,4}&=c_2u_2\lambda_1^{4,2}
                   +c_1u_3\lambda_1^{4,3}-u_2^2\lambda_1^{4,4},\\
 V_{2,4}&=c_2u_2\lambda_2^{4,2}
                   +c_1u_3\lambda_2^{4,3}-u_2^2\lambda_2^{4,4}
                   +h_3\widehat q_{2,1},\\
 V_{3,4}&=c_1u_3\lambda_3^{4,3}
                   +h_3\widehat q_{3,1}+h_2\widehat q_{3,2},\\
 V_{4,4}&=h_3\widehat q_{4,1}+h_2\widehat q_{4,2}
                                  +h_1\widehat q_{4,3}.
\end{aligned}
\end{equation}
The quotient coefficients are those defined in
\eqref{history_quotient_vector}.  When $m=3$, the first row expresses the degree condition $h_4=0$,
and the corresponding right-hand side is zero.

\begin{proof}[Verification of \eqref{mregfourthnonzero}--\eqref{mregfourthzero}]
For $b\ge4$ write the rows as
$\ell_rc_4-h_4+\rho_ru_4=V_{r,4}$, where
$\ell_r=\lambda_r^{4,0}$ and $\rho_r=-\lambda_r^{4,4}$.
Any three independent rows form an invertible subsystem of
$\mathsf M_4\mathbf x_4=\mathbf v_4$ and determine the candidate uniquely.  The remaining equation is equivalent
to $\Phi_4=0$, where the following explicit row combination eliminates
all three current coefficients.  If $i<k<s$ are the rows other than
$r$, put
\[
 \omega_r=(-1)^{r-1}
 \bigl((\ell_k-\ell_i)(\rho_s-\rho_i)
             -(\ell_s-\ell_i)(\rho_k-\rho_i)\bigr),
 \qquad
 \Phi_4=\sum_{r=1}^{4}\omega_rV_{r,4}.
\]
Let $\boldsymbol\omega=(\omega_1,\ldots,\omega_4)^{\mathsf T}$,
so that $\Phi_4=\boldsymbol\omega^{\mathsf T}\mathbf v_4$.
Its entries are signed $3\times3$ minors; hence
$\boldsymbol\omega^{\mathsf T}\mathsf M_4=0$ and
$\boldsymbol\omega\ne0$ by Lemma~\ref{rank}.
This also proves the claimed equivalence for $b\ge4$.

Insert \eqref{mregsecond} and \eqref{c3h3regular} in the four explicit
entries \eqref{fourth_rhs}.  The possible monomials in the nonzero-scale
case are $t^4,t^2\xi,\xi^2$.  Their coefficients in the displayed row
combination are
\begin{equation}\label{fourth_coefficient_table}
\begin{array}{c|c}
 \text{monomial}&\text{coefficient in }\Phi_4\\ \hline
 t^4&0\\
 t^2\xi&4\Lambda_0\Theta_bz(z+2)\\
 \xi^2&\Lambda_0\Delta_3^2\Psi_4
\end{array}
\end{equation}
with
\begin{equation*}
 \begin{split}
 \Lambda_0={}&-\frac{135m^3(n+2)(m-2)(m-1)^2(m+1)^2(m+2)}
 {b^3(b-3)(b-2)(b-1)^3(b+1)^3(b+2)(b+3)}\\
 &\times\frac{(n+1)^2(z+2)}
 {(n-3)(n-2)(n-1)n^5(z-2)(z-1)(z+1)^4}.
\end{split}
\end{equation*}
Here the zero entry also follows without expansion: at $\xi=0$ the
four right-hand sides are those of the known reference solution.
For the quadratic coefficient, set $t=0$.  The coefficient of
$\sigma^2$ in \eqref{fourth_rhs} is the four-component vector with
entries
\[
 \alpha_2\gamma_2\lambda_r^{4,2}
   -\gamma_2^2\lambda_r^{4,4}\quad(r=1,2),
 \qquad \beta_2(\alpha_2\lambda_r^{2,0}-\beta_2)\quad(r=3,4).
\]
Multiplication by $\boldsymbol\omega^{\mathsf T}$ gives
$\Lambda_0\Psi_4$ after substituting
\eqref{alpha2beta2} and \eqref{lambda}.  Multiplying by $\Delta_3^2$
gives the $\xi^2$ entry in \eqref{fourth_coefficient_table}.
The $t^2\xi$ entry is obtained by taking in each product of
\eqref{fourth_rhs} one reference factor and one of the deviations in
\eqref{mregsecond} or \eqref{c3h3regular}; the reference coefficients
are explicitly \eqref{reference_low_explicit}.  Thus
\eqref{fourth_coefficient_table} follows by taking the displayed
linear combination of the four rows; it imposes no additional condition
on the parameters.
It gives
\[
 \Phi_4=\Lambda_0\xi
       (\Delta_3^2\Psi_4\xi+4\Theta_bz(z+2)t^2)
 \quad(t\ne0),\qquad
 \Phi_4=\Lambda_0\Psi_4\sigma^2\quad(t=0).
\]
Every numerator and denominator factor of $\Lambda_0$ is nonzero
for $b\ge4$: in addition to the bounds following \eqref{vartheta3def},
$n\ge5$, $n+2<p$, $m+2<p$, and $b+3<p$.  This proves the result
in the stable range.

For $b=3$, multiply the last generic row by $b-3$ and use
Appendix~\ref{app:degreeboundary}.  The fourth equation becomes a nonzero
multiple of $u_4=0$; the normalized row combination has multiplier
\[
 \left.(b-3)\Lambda_0\right|_{b=3}
 =-\frac{m^2(m-1)(m+2)^2(m+3)(m+5)}
          {3072(m+1)^4(m+4)^4}\ne0.
\]
Here $3\le m\le p-7$, so all displayed factors are units.
The two necessary conditions therefore have the same form at this
boundary.
\end{proof}

\subsection{The fifth-level nonzero-scale calculation}\label{app:fifth}

\begin{proof}[Fifth-level nonzero-scale calculation]
For $b\ge4$, we establish the fifth-level conditions used in
Lemma~\ref{branchscreening}(i).  Only the candidate
\eqref{mregxicandidate} needs to be excluded.  We first calculate in the
stable range $b\ge5$, and then specialize to $b=4$ as described below.
A common dilation by $\kappa\in\F^*$ multiplies every coefficient
of index $s$ by $\kappa^s$.  Taking $\kappa=t^{-1}$ therefore
normalizes $t$ to $1$.
The first three states are given by \eqref{reference_low_explicit},
\eqref{mregsecond}, and \eqref{c3h3regular}, with $\xi$ fixed by
\eqref{mregxicandidate}.
Compute the four entries \eqref{fourth_rhs}.  They satisfy the
fourth-level compatibility condition.  To obtain the unique state
$\mathbf x_4$, evaluate \eqref{stable_candidate_formula} at $j=4$,
with $I_0=(1,2,3)$, $I_1=(2,3,4)$ and $d=-2(z+2)$.
This specifies $c_4,h_4,u_4$ even when one of the two consecutive
minors vanishes; no fixed invertible triple is assumed.
These fourth coefficients must be retained in the next right-hand
side.  Substituting them into the matrices $\mathsf Q_5$ and
$\mathsf R_5$ of Proposition~\ref{largegaprecursion} gives the explicit
vector $\mathbf v_5$.  We test this vector against two consecutive $3\times3$ subsystems
of $\mathsf M_5$.

Next evaluate \eqref{stable_candidate_formula} at $j=5$, using the
triples $(2,3,4)$ and $(3,4,5)$ and $d=-2(z+1)$.
Denote this candidate by $\mathbf y_5$ and form
\[
 E_{1,5}=(\mathsf M_5)_{1,*}\mathbf y_5-V_{1,5},\qquad
 E_{2,5}=(\mathsf M_5)_{2,*}\mathbf y_5-V_{2,5}.
\]
Every actual fifth-level solution must make both expressions zero.
No fifth-level free parameter is introduced.  Define
\begin{equation}\label{U5}
 \mathcal U_5=-\frac{192m(m+b)(z+3)(m-2)(m-1)(m+b-1)z^2(z+2)^2}
 {(b+3)n^5(z+1)^5\Delta_3^3\Psi_4^2},
\end{equation}
and normalize by
\[
 \widehat E_{1,5}=\frac{(b+4)(z+1)}{\mathcal U_5}E_{1,5},
 \qquad
 \widehat E_{2,5}=\frac{E_{2,5}}{\mathcal U_5(m+1)}.
\]
All factors used as divisors are units: $\Delta_3\Psi_4\ne0$ by
\eqref{mregxicandidate}, while the remaining positive integer factors
are at most $p-1$ under \eqref{rangealllarge}.

Substitution of the specified fourth-level solution into
$\mathbf v_5=\mathsf Q_5\mathbf h_5+\mathsf R_5\mathbf c_5$,
followed by the two displayed matrix eliminations, gives the following
coefficient table.  It is convenient
to use $m=z-2b+3$, so the entries are polynomials in $z$ over
$\mathbb Z[b]$:
\begin{equation}\label{fifthcoefficienttable}
\begin{array}{c|l|l}
 &\widehat E_{1,5}&\widehat E_{2,5}\\ \hline
 z^5&3b+5&0\\
 z^4&-6b^2+22b+60&0\\
 z^3&-62b^2-89b+115&1\\
 z^2&12b^3+132b^2-216b-600&9\\
 z&-116b^3+584b^2+712b-2220&-2b-16\\
 1&-40b^3+192b^2+360b-1728&20b-108
\end{array}
\end{equation}
The table is obtained from the explicit vector $\mathbf v_5$ in
Proposition~\ref{largegaprecursion}, the candidate formula
\eqref{stable_candidate_formula}, and the displayed normalizations.

Recall $\Delta_0,F_5,G_5$ from
Lemma~\ref{branchscreening}(i), and define the auxiliary
polynomial $\Psi_5$ by
\begin{equation}\label{FGstruct}
\begin{aligned}
 \Delta_0&=z^2+2z-18,\\
 F_5&=(z+6)\Delta_0+(m-3)(z-10),\\
 G_5&=(z+14-m)\Delta_0+(m-3)(4z-37),\\
 \Psi_5&=(3b+5)z^2-(6b^2+5b-15)z-2b^2+14b+60.
\end{aligned}
\end{equation}
Expanding with $m=z-2b+3$ gives
$F_5=z^3+9z^2-(2b+16)z+20b-108$, the second column of
\eqref{fifthcoefficienttable}.  Multiplication by $\Psi_5$ and subtraction
of $8(b+3)G_5$ gives its first column.  Thus
\[
 \widehat E_{2,5}=F_5,\qquad \widehat E_{1,5}=\Psi_5F_5-8(b+3)G_5.
\]
Since $8(b+3)$ is a unit, both residuals can vanish only if
\begin{equation}\label{vfifthres}
 F_5=G_5=0.
\end{equation}
For $b=4$, the same rational identities are specialized as in
Appendix~\ref{app:degreeboundary}, with $u_5=0$.  All final denominators in
\eqref{U5} are still units, and the formulas contain no surviving factor
$b-4$ in a denominator.  These are necessary boundary conditions, not
an application of the rank lemma outside its range.

Finally, direct multiplication of the expressions in \eqref{FGstruct}
gives the integer polynomial identity
\begin{equation*}
 (2z^2-4z+38)F_5-(20-2z)G_5
       =2(z+1)(z+2)^2\Delta_0.
\end{equation*}
The right-hand multiplier is a unit, so \eqref{vfifthres} forces
$\Delta_0=0$.  The two remaining equations are
$(m-3)(z-10)=(m-3)(4z-37)=0$; subtracting four times the first from the
second gives $3(m-3)=0$.  However
$\Delta_3=\Delta_0-3(m-3)$, contrary to $\Delta_3\ne0$.  This completes the fifth-level exclusion
used in Lemma~\ref{branchscreening}(i).
\end{proof}

\subsection{The sixth-level even calculation}\label{app:even}

\begin{proof}[Sixth-level even calculation]
We verify the sixth-level conditions used in
Lemma~\ref{branchscreening}(ii).  The fourth-level condition \eqref{mregfourthzero} forces $\Psi_4=0$.
For $b=3,4$, $\Psi_4$ is respectively $(z-1)(z+2)$ and $z(z+1)$,
and neither can vanish in the prescribed range.  Hence $b\ge5$,
so $p\ge17$.  Combining $\Psi_4=0$ with $z=m+2b-3$ gives, in $\F$,
\begin{equation}\label{q4curve}
 b=4-\frac{z(z+1)}2,\qquad m=z^2+2z-5,
\end{equation}
and consequently $\Delta_2=4-3z$ and
$\Delta_3=-2(z-1)(z+3)\ne0$.  Since $t=0$ and $\Delta_3\ne0$, \eqref{mregthird} and the invertibility
of $\mathsf M_3$ give $\mathbf x_3=0$.

The fourth-level right-hand side is $\sigma^2$ times the
four-component vector computed in Appendix~\ref{app:fourth}.  The first
three rows form an invertible subsystem: on $\Psi_4=0$,
$\Pi_4(m)=-2z(z+2)\ne0$, so their determinant is a unit by \eqref{detfactor}.
Solving the rows $1,2,3$ of $\mathsf M_4\mathbf x_4=\mathbf v_4$ gives
\begin{equation}\label{even_fourth_state}
 \mathbf x_4=\sigma^2\mathbf d_4,\qquad
 \mathbf d_4:=(\alpha_4,\beta_4,\gamma_4)^{\mathsf T},
\end{equation}
where the following factored expressions record the solution needed
at level six:
\begin{equation}\label{even_fourth_coefficients}
\begin{aligned}
 \alpha_4&=
 \frac{(z-1)(z+4)(z+5)(n-3)(n-1)\mathcal A_4(z)}
 {320(z-2)(z+1)^4(z+3)^2b^2(b+1)^2n^3},\\
 \beta_4&=-
 \frac{(z+4)(m-2)(m-1)m(m+1)\mathcal B_4(z)}
 {20(z-2)(z+1)^4(z+3)^2b^2(b+1)^2n^4},\\
 \gamma_4&=-
 \frac{(z-1)(b-2)\mathcal C_4(z)}
 {20(z-2)(z+1)^4(z+3)b(b+1)n^4},\\
 \mathcal A_4(z)&=9z^4+111z^3+356z^2-952,\\
 \mathcal B_4(z)&=9z^4+36z^3-49z^2-170z+198,\\
 \mathcal C_4(z)&=81z^5+51z^4-1648z^3+288z^2+8872z-7776.
\end{aligned}
\end{equation}
These are obtained by substituting
$c_2=\alpha_2\sigma$, $h_2=\beta_2\sigma$, $u_2=\gamma_2\sigma$
in \eqref{fourth_rhs}, dividing by $\sigma^2$, and solving the $3\times3$ subsystem on rows $1,2,3$.
The simplification uses $z^2+z=8-2b$ and $z^2+3z-6=2n$;
no square root of $\sigma$ or dilation setting $\sigma=1$ is used.
The fourth equation is precisely the already imposed condition $\Psi_4=0$.
All denominators in \eqref{even_fourth_coefficients} are units.

The nonzero preceding indices are now even, so
Corollary~\ref{supportpropagation} yields $\mathbf x_5=0$.
We first take $b\ge6$.  At level six,
Proposition~\ref{largegaprecursion} gives
\[
 \mathbf v_6=\mathsf Q_6\mathbf h_6-\mathsf R_6\mathbf c_6.
\]
After substituting \eqref{mregsecond}, \eqref{even_fourth_state}, and
$\mathbf x_3=\mathbf x_5=0$, every entry has the common factor
$\sigma^3$.  We test the resulting vector directly against two
consecutive $3\times3$ subsystems of $\mathsf M_6$.

For $b\ge6$ use the two triples $3,4,5$ and $4,5,6$ in
\eqref{stable_candidate_formula} at $j=6$; its denominator is $-2z$.
Let $\mathbf y_6$ be that candidate, and set
$E_{r,6}=(\mathsf M_6)_{r,*}\mathbf y_6-V_{r,6}$ for $r=2,3$.
Define
\begin{equation*}
 \Omega=-\frac{3456(z-1)(z+4)^2(m-2)(m-1)m(m+1)(z^2+3z-2)}
 {(z-2)^3(z+3)(z^2+z-14)(z^2+z-10)^3
                    (z^2+z-8)^3n^6(z+1)^6}.
\end{equation*}
The two necessary residuals are normalized as
\[
 \widehat E_{3,6}=\frac{E_{3,6}}{\sigma^3\Omega(z-1)},\qquad
 \widehat E_{2,6}=\frac{z(z+3)(z^2+z-16)}{\sigma^3\Omega}E_{2,6}.
\]
All of these divisions are valid.  On \eqref{q4curve} the four
quadratics $z^2+z-14$, $z^2+z-10$, $z^2+z-8$, and $z^2+z-16$
are $-2(b+3),-2(b+1),-2b,-2(b+4)$, respectively.
Also $z^2+3z-2=2(m+b)\ne0$.
The integer bounds give $0<z-2<z+3<p$.  Finally $z+4=0$ together
with $\Psi_4=0$ would imply $2(b+2)=0$, which is impossible.
Thus $\Omega$ and all the stated normalizing factors are units.

The explicit solution \eqref{even_fourth_coefficients} and the direct
evaluation of $\mathbf v_6$ give, after these row operations,
\begin{equation}\label{even_residual_polynomials}
\begin{aligned}
 \widehat E_{3,6}&=z^2-5z-26,\\
 \widehat E_{2,6}&=3z^7-8z^6-160z^5+20z^4+1721z^3
                           +52z^2-4796z+3744.
\end{aligned}
\end{equation}
Recall the polynomials used in Lemma~\ref{branchscreening}(ii):
\begin{equation}\label{vsixthres}
 K_6=z^2-5z-26,\qquad S_6=57z+182.
\end{equation}
Polynomial division of the second line of
\eqref{even_residual_polynomials} by $K_6$ gives
\[
 \widehat E_{2,6}=(3z^5+7z^4-47z^3-33z^2+334z+864)K_6+144S_6.
\]
Hence an actual solution requires $K_6=S_6=0$, because $144$ is a unit.
For $b=5$, the same necessary identities follow from the one-step
specialization in Appendix~\ref{app:degreeboundary}, with $u_6=0$.  The final denominators
remain units: $z^2+z=-2$ makes the four quadratics above
$-16,-12,-10,-18$, and $p\ge17$.  Thus this boundary is also covered.

The contradiction is the integer polynomial identity
$57^2K_6-(57z-467)S_6=520$.  For example, its two products have the same quadratic and linear terms,
$3249z^2-16245z$, and constant terms $-84474$ and $-84994$.
Their difference is $520=2^3\cdot5\cdot13$, nonzero for $p\ge17$.
No division by $57$ is used.  This excludes every state in the even
branch.
\end{proof}

\subsection{The pure-third sixth-level obstruction}\label{app:purethird}

\begin{proof}[Pure-third sixth-level calculation]
We supply the obstruction for $b\ge5$ used in
Lemma~\ref{branchscreening}(iii).  If $u_3=0$, then $t=\sigma=0$ makes the third-level right-hand side
zero; the nonzero $c_3,h_3$ minor in the main text gives
$\mathbf x_3=0$, including on $\Delta_3=0$.  All preceding states of positive index then vanish.
At each $4\le j\le b$, Proposition~\ref{largegaprecursion} gives
$\mathbf v_j=0$ and Corollary~\ref{uniquecontinuation} gives
$\mathbf x_j=0$.  Thus $F_B=X^b$, contrary to distinctness of its roots.
We may therefore assume $v:=u_3\ne0$; \eqref{mregthird} forces
$\Delta_3=0$.

Solving the two independent third equations gives
\begin{equation}\label{thirdprofileparameters}
 c_3=\alpha v,\quad h_3=\beta v,\qquad
 \alpha=\frac{(z+6)(z+8)n(n-1)}{8b(b+1)(z-6)(z-4)},\quad
 \beta=-\frac{3m(m^2-1)}{2b(b+1)(z-6)(z-4)}.
\end{equation}
These expressions follow directly by substituting $\Delta_3=0$ into
the first two rows of the third-level system.
At level four no lower-order index sum is possible, and then the same
is true at level five.  Corollary~\ref{supportpropagation}, applied
successively, gives $\mathbf x_4=\mathbf x_5=0$.
If $b=5$, this already gives
$F_B=X^5+vX^2=X^2(X^3+v)$, which has a repeated root.  This excludes $b=5$ without using level six.

Now let $b\ge6$.  Since $\mathbf x_4=\mathbf x_5=0$, direct
evaluation of
\[
 \mathbf v_6=\mathsf Q_6\mathbf h_6-\mathsf R_6\mathbf c_6
\]
shows that the relevant rows $r=3,4,5,6$ are $V_{r,6}=v^2f_r$, where
\begin{equation}\label{purethirdprofile}
 f_r=\begin{cases}
 \alpha\lambda_{b-3}(6,3)-\lambda_{b-3}(6,6),&r=3,\\[1mm]
 -\beta(\alpha\lambda_{b-r}(3,0)+\beta),&r=4,5,6.
 \end{cases}
\end{equation}
These are the only components of $\mathbf v_6$ needed below.

Solve the last three of these equations, whose matrix is
$\mathsf N=\mathsf M_{6,(4,5,6)}$.  On $\Delta_3=0$, \eqref{detfactor} gives
\[
 \det\mathsf N=\Gamma_6(m+3)\Pi_6(m+3),\qquad
 \Pi_6(m+3)=-6(z+5)\ne0.
\]
The last inequality follows from
\[
 (z+5)(z-6)=-6(b+2),\qquad
 (z+6)(z-7)=-6(b+4):
\]
$b+2,b+4<p$ and $p\ge17$, so in particular $z+5,z+6\ne0$.
For $\mathbf y=(y_3,y_4,y_5,y_6)^{\mathsf T}$, define the linear
functional
\[
 \mathcal L_6(\mathbf y):=(\mathsf M_6)_{3,*}\mathsf N^{-1}
      (y_4,y_5,y_6)^{\mathsf T}-y_3.
\]
It is the residual in row three after rows four through six have been
solved.  Set $\mathbf f=(f_3,f_4,f_5,f_6)^{\mathsf T}$, with entries
from \eqref{purethirdprofile}.  An actual solution requires
$v^2\mathcal L_6(\mathbf f)=0$.

Put $D_r=\lambda_{b-r}(3,0)$ for $r=3,4,5,6$, and define
\[
 \mathbf D=(D_3,D_4,D_5,D_6)^{\mathsf T},\qquad
 \mathbf e=(1,0,0,0)^{\mathsf T},\qquad
 \mathbf 1_4=(1,1,1,1)^{\mathsf T}.
\]
The vector $\mathbf f$ decomposes as
\[
 \mathbf f=-\alpha\beta\mathbf D-\beta^2\mathbf 1_4+\chi\mathbf e,\qquad
 \chi=f_3+\beta(\alpha\lambda_{b-3}(3,0)+\beta).
\]
The all-ones vector is the negative of the constant column of the
coefficient matrix, so
$\mathcal L_6(\mathbf 1_4)=0$.  Also $\mathcal L_6(\mathbf e)=-1$, since its last
three entries are zero.  Solving the indicated $3\times3$ subsystem with the entries $D_4,D_5,D_6$, and then substituting into row three, gives
\[
 \mathcal L_6(\mathbf D)=\frac{2(z-6)(z^2+11z+48)}
                       {(b+3)n(n-1)(n-2)}.
\]
The four coefficients are given by \eqref{lambda}, and the
simplification uses $6(b-3)=z-z^2$.  Thus only these two scalar
residuals, not a general sixth-level state, are needed.

For completeness the last cancellation can be checked term by term.
Set
\begin{equation}\label{vK6}
 \widetilde\kappa_6=
 \frac{27(z+6)^2m(m^2-1)}
 {(z-6)(z-4)^2(z-3)b^2(b+1)^2(b+3)}.
\end{equation}
Using \eqref{thirdprofileparameters} and $\Delta_3=0$ in the two
nonconstant contributions gives
\[
 \frac{-\alpha\beta\mathcal L_6(\mathbf D)}{\widetilde\kappa_6}
       =\frac{z^2+11z+48}{12(z+6)},\qquad
 \frac{-\chi}{\widetilde\kappa_6}
       =-\frac{z^2-z-24}{12(z+6)}.
\]
The difference of the two numerators is $12(z+6)$, so
$\mathcal L_6(\mathbf f)=\widetilde\kappa_6$.  All its denominator factors are
units: $z\ge12$, $z+3\le p-1$, and $b+3,n,m+1<p$.
Its numerator is also nonzero, including $z+6$ as checked above.
Consequently $v^2\mathcal L_6(\mathbf f)\ne0$, the required contradiction.
\end{proof}

\subsection{The cubic endpoint calculation}\label{app:cubic}

\begin{proof}[Cubic endpoint calculation]
This is the $b=3$ calculation used in
Lemma~\ref{branchscreening}(i).  Here $n=m+1$ and $z=m+3$, with
$3\le m\le p-7$.  All endpoint rows below are obtained from
\eqref{actual_degree_remainder}--\eqref{actual_degree_residual}, not
from Proposition~\ref{largegaprecursion} at $j=4,5$.
If $t=0$, the fourth-level boundary condition has
$\Psi_4=(z-1)(z+2)\ne0$, so $\sigma=0$.
Also $\Delta_3=z(z-1)\ne0$, so $u_3=0$.  Hence $F_B=X^3$, a
contradiction.  Thus $t\ne0$.

If $\xi=0$, the first three coefficients already give
$F_B=F_{B^0}$.  Otherwise normalize $t=1$ and use
\eqref{mregxicandidate}, which becomes
\begin{equation}\label{cubicxicandidate}
 \xi=-\frac{64}{z(z-1)^3}.
\end{equation}
We compute further using only the degree-three reduction, so
$u_4=u_5=0$.  At level four, the raw equations with $X$-degrees $i=0,1$ have
coefficient matrix
\[
 \begin{pmatrix}1&-1\\-(n-2)/2&(n+2)/2\end{pmatrix},
 \qquad \det=2.
\]
With $U_i=\mathcal W_{i,4}$ for $i=0,1$, solving this
$2\times2$ subsystem gives
\[
 c_4=\frac{U_1}{2}+\frac{n+2}{4}U_0,\qquad
 h_4=\frac{U_1}{2}+\frac{n-2}{4}U_0.
\]
Here the rows retain their raw normalization: their $h_4$
coefficients are $-g_{0,0}=-1$ and $-g_{1,0}=(n+2)/2$.
The values $U_i$ are therefore raw right-hand sides, as defined in
\eqref{actual_degree_rhs}, rather than entries of a vector $\mathbf v_4$.
This supplies the preceding state needed at level five.

At level five the three raw equations, in order $i=0,1,2$, have the
augmented matrix
\begin{equation}\label{cubicaugmented}
 \begin{pmatrix}
 -1&-1&W_0\\[1mm]
 (n-3)/2&(n+2)/2&W_1\\[1mm]
 -(n-5)(n-2)/8&-n(n+3)/8&W_2
 \end{pmatrix}.
\end{equation}
These are the rows $i=0,1,2$ of the actual-degree equations
\eqref{actual_degree_residual}, with $W_i=\mathcal W_{i,5}$.
Their second column is $-g_{i,0}$, so no normalization by $g_{i,0}$
has been suppressed.  If $c_5$ or $h_5$ is beyond its actual degree,
its value is fixed to zero.  Allowing it formally in
\eqref{cubicaugmented} only enlarges the candidate system; every actual
solution is still included.  In particular, at $n=4$ the displayed
polynomial continuation of the $c_5$ column is harmless because an
actual solution has $c_5=0$.

The first two rows have determinant $-5/2\ne0$.  Solving them gives
\[
 c_5=-\frac{2W_1+(n+2)W_0}{5},\qquad
 h_5=\frac{2W_1+(n-3)W_0}{5}.
\]
Substitution in the third row leaves exactly the condition
\begin{equation}\label{cubic_remaining_row}
 W_2+\frac{n-1}{2}W_1+\frac{n^2-n-4}{8}W_0=0.
\end{equation}
We record a coefficient table for this finite substitution, so the
last cancellation can be checked directly.  Put
$D=5040(m+1)^5(m+2)^4(m+4)^5$ and $T=m(m-1)(m-2)$.
Inserting \eqref{cubicxicandidate}, the initial states, and the two
fourth-level values into \eqref{actual_degree_remainder} and
\eqref{actual_degree_rhs}, with $b=3$ and $j=5$, gives
\[
\begin{aligned}
 W_0&=-\frac{28T}{D}P_0(m),\qquad
 W_1=\frac{7T(m-3)}{D}P_1(m),\\
 W_2&=\frac{T(m-4)(m-3)(m+1)}{D}P_2(m),
\end{aligned}
\]
where the full polynomials are specified by their coefficients:
\begin{equation*}
 \begin{array}{c|r|r|r}
 k &[m^k]P_0 &[m^k]P_1 &[m^k]P_2\\ \hline
10&15&15&0\\
9&15&25&0\\
8&-2035&-2605&175\\
7&-17875&-28825&3150\\
6&-54100&-133850&21665\\
5&41820&-253720&54530\\
4&764880&246560&-129080\\
3&2243392&2221184&-1287840\\
2&2979200&4430464&-3703552\\
1&1817856&3897344&-4894592\\
0&396288&1296384&-2535936
\end{array}
\end{equation*}
Combining the three polynomials specified by the table gives the identity
\[
\begin{split}
 &2(m-4)(m-3)(m+1)P_2
       +7m(m-3)P_1-7(m^2+m-4)P_0\\
 &\hspace{20mm}=-69120\prod_{k=1}^{6}(m+k).
\end{split}
\]
Since $n=m+1$, the left-hand side of \eqref{cubic_remaining_row} is
$T/(2D)$ times this polynomial.  After cancellation, it is
\begin{equation}\label{cubicfifth}
 -\frac{48m(m-2)(m-1)(m+3)(m+5)(m+6)}
 {7(m+1)^4(m+2)^3(m+4)^4}.
\end{equation}
The bounds $3\le m\le p-7$ and $p\ge11$ show that all its numerator
and denominator factors are units.  It cannot be zero.  Thus the
nonzero candidate has no fifth-level continuation, and only the
reference cubic remains.
\end{proof}

\subsection{The quartic endpoint calculation}\label{app:quartic}

\begin{proof}[Quartic endpoint calculation]
We verify the $b=4$ endpoint in Lemma~\ref{branchscreening}(iii).
The fourth-level condition has $\Psi_4=z(z+1)\ne0$, so $\sigma=0$.
If $u_3=0$, the initial states and then the fourth state are zero;
this gives $F_B=X^4$, impossible.  Hence $v:=u_3\ne0$ and
\begin{equation}\label{quarticDelta3}
 \Delta_3=m^2+9m+26=0.
\end{equation}
The first two third-level equations give
\begin{equation}\label{quarticc3h3}
 c_3=\frac{m+13}{40}v,\qquad h_3=-\frac{3m}{40}v.
\end{equation}
Since $\mathbf v_4=0$, Corollary~\ref{uniquecontinuation} gives
$\mathbf x_4=0$.

Here the actual root degree is four, so $u_5=u_6=0$ throughout the
next calculation.  We use \eqref{actual_degree_residual} at levels
five and six.  Its fifth-level historical right-hand sides
$\mathcal W_{i,5}$ vanish: all surviving positive-index data have index
three, whereas these right-hand sides have weight five.  The two raw
columns on $c_5,h_5$, in $X$-degrees $0,1$, have matrix
\[
 \begin{pmatrix}-1&-1\\(n-3)/2&(n+2)/2\end{pmatrix},
 \qquad n=m+2,
\]
with determinant $-5/2$.  Thus $c_5=h_5=0$.  If a coefficient is
past its degree, it is already zero; the same two-column elimination
is a valid necessary calculation when the vanishing coefficient is formally retained as an unknown.  No use of a generic rank-three $\mathsf M_5$ is made.

At level six only the quadratic contribution of the third-level state
remains.  Evaluate \eqref{actual_degree_remainder} and
\eqref{actual_degree_rhs} with $b=4$, $u_5=u_6=0$ and $c_5=h_5=0$,
and reduce using $m^2=-9m-26$.  Writing
$W_i=\mathcal W_{i,6}/v^2$, the actual-degree system is
$A_i c_6+B_i h_6=W_i v^2$, where $B_i=-g_{i,0}$ and
\begin{equation}\label{quartic_sixth_table}
\begin{array}{c|c|c|c}
 i&A_i&B_i&W_i\\ \hline
 0&1&-1&3(31m+52)/1600\\
 1&-(m-2)/2&(m+4)/2&93(5m+26)/4480\\
 2&-(7m+11)/4&(m+8)/4&351(2m+25)/22400\\
 3&-13(m+3)/4&-(m-2)/8&-13(84m-97)/22400
\end{array}
\end{equation}
All entries are raw coefficient comparisons, not normalized by a
possibly vanishing subminor.  Formal coefficients beyond degree are
again fixed to zero for an actual solution; allowing them while
eliminating only enlarges the system.

The first two rows have determinant $3$ and give
\begin{equation}\label{quartic_sixth_solution}
 c_6=\frac{3(9m-26)}{11200}v^2,\qquad
 h_6=-\frac{39(8m+15)}{5600}v^2.
\end{equation}
Substitution in rows $i=2,3$ gives the residuals
\begin{equation}\label{quarticsix}
 A_2c_6+B_2h_6-W_2v^2=-\frac{351}{3200}v^2,
 \qquad
 A_3c_6+B_3h_6-W_3v^2=\frac{13(27m+4)}{6400}v^2.
\end{equation}
Both must be zero.  All denominators have prime factors at most seven,
whereas $p\ge13$.  Since $v\ne0$ and $351=27\cdot13$, the first
condition forces $p=13$.  Now $3\le m\le p-2b-1=4$;
\eqref{quarticDelta3} excludes $m=3$ and leaves $m=4$.
Finally \eqref{quarticc3h3} and \eqref{quartic_sixth_solution}, with
these parameters and the factors $(-2)^s$ restored in $F_{A^c}$,
give exactly \eqref{quartictriple}.  In particular the displayed
$h_6$ is zero, as the actual degree $m=4$ requires.
\end{proof}

\end{document}